\documentclass{amsart}[11 pt]
\usepackage{amsmath}
\usepackage{amssymb}
\usepackage{amsfonts}
\usepackage{amscd}
\usepackage{stmaryrd}
\newtheorem{theorem}{Theorem}[section]
\newtheorem{lemma}[theorem]{Lemma}
\newtheorem{proposition}[theorem]{Proposition}

\newtheorem{remark}[theorem]{Remark}
\newtheorem{definition}[theorem]{Definition}

\theoremstyle{corollary}
\theoremstyle{lemma}
\theoremstyle{proposition}
\theoremstyle{definition}
\theoremstyle{remark}
\theoremstyle{theorem}
\numberwithin{equation}{section}

\def\Xint#1{\mathchoice
	{\XXint\displaystyle\textstyle{#1}}%
	{\XXint\textstyle\scriptstyle{#1}}%
	{\XXint\scriptstyle\scriptscriptstyle{#1}}%
	{\XXint\scriptscriptstyle\scriptscriptstyle{#1}}%
	\!\int}
\def\XXint#1#2#3{{\setbox0=\hbox{$#1{#2#3}{\int}$}
		\vcenter{\hbox{$#2#3$}}\kern-.5\wd0}}

\def\dashint{\Xint-}
\usepackage{graphicx}
\usepackage{color}
\begin{document}
\title[Homogenization of non-convex energies]{Homogenization of non-convex integral energies with Orlicz growth via periodic unfolding}
\author{Joel Fotso Tachago$^{\ddagger }$}
\curraddr{$^{\ddagger }$University of Bamenda, Higher Teachers Training
College, Department of Mathematics, P.O. Box 39, Bambili, Cameroon}
\email{fotsotachago@yahoo.fr}
\author{Guiliano Gargiulo$^{\intercal }$}
\curraddr{$^{\intercal }$Dipartimento di Scienze e Tecnologie (DST)
 University of Sannio, Via dei Mulini (CUBO), Benevento, 82100, Italy }
\email{ggargiul@unisannio.it}
\author{Hubert Nnang$^{\sharp }$}
\curraddr{$^{\sharp }$ University of Yaounde I, \'{E}cole Normale Sup\'{e}%
rieure de Yaound\'{e}, P.O. Box 47 Yaounde, Cameroon.}
\email{hnnang@uy1.uninet.cm}
\author{Elvira Zappale $^{\#}$}
\curraddr{$^{\#}$ Dipartimento di Scienze di Base ed Applicate per l'Ingegneria, Sapienza-Universit\'a di Roma, Via A. Scarpa, 16, 00161, Rome, Italy}
\email{elvira.zappale@uniroma1.it}

\date{June, 2024}
\maketitle

\begin{abstract}
The	periodic unfolding method is extended to the Orlicz setting and used to prove a homogenization result for non-convex integral energies defined on vector-valued configurations in the Orlicz-Sobolev setting.


\end{abstract}

\section{Introduction\label{sec1}}

Homogenization of periodic structures via two-scale convergence was introduced in \cite{All1,ngu0} to deal with many problems formulated in terms of partial differential equations and integral functionals in Lebesgue and Sobolev spaces and $BV$ ones (see \cite{A, bbs, cdmsz1, cdmsz2, FF1, FF2}. Indeed the importance of detecting the overall behaviour of a material which might include periodically distributed heterogeneties is very important in many applications, from nonlinear elasticity, mean-field games, to micro and ferromagnetic, conductivity, evolutions problems, polycrystals, discrete models, etc. Indeed we refer to \cite{AMMZ, BMZ, BD, DDI, FGY, FFV1, FFV2, RBCM} among a much wider bibliography.
 
Quite recently the two-scale theory has been extended to the Orlicz setting (see \cite{fotso nnang 2012}), and applied in many
types of problems (see \cite{fotso23,FNZIMSE,FNZOpuscula,FTGNZ,kenne Nnang},  since in many contexts, polynomial growth cannot be considered. 

On the other hand the theory of two-scale convergence has been rephrased by \cite{CDG1, CDG2, CDG3} in terms of periodic unfolding, we refer to \cite{CDG4} for a complete theory in the Sobolev setting. However not many results dealing with integral functionals have been obtained in this context: we refer to \cite{fotso nnang 2012, Elvira 1, V, bbs}, among a wider bibliography, and, in most of these papers explicit homogenized limiting energies have not been computed if the original model is not convex, the first result in the non-convex setting is the one obtained in \cite{ciora1}, on the other hand quasiconvexity in the sense of Morrey have been imposed therein.

Now we aim at extending this latter result to the framework of Orlicz-Soboelv spaces and at a complete removal of any convexity or quasiconveity assumption on the integrand $f$ considered below.
\color{black}
Namely, we are interested in the asymptotic behaviour of
nonconvex integral functionals modeling homogenization problems in the Orlicz setting. To be precise let $N, d \in \mathbb N$. 
Let $Y:=\left] 0,1\right[^N$ and $\mathcal{A}_{0}$ the class of all
bounded open subsets of $
\mathbb{R}
^{N}$ with Lipschitz boundary. For each $\Omega \in \mathcal{A}_{0}$ and
every $\varepsilon \in [ 0,+\infty)$,  converging to $%
0,$ we consider the family of
functionals: 
\begin{equation}\label{funct}
u\in W^{1}L^B(\Omega ;
\mathbb{R}^d) \mapsto \int_{\Omega }f\left( \frac{x}{\varepsilon},\nabla
u\right) dx,
\end{equation} where

\begin{equation}\label{tch1}
\begin{tabular}{l}
$f:\left( y,\xi\right) \in 
\mathbb{R}
^N\times 
%TCIMACRO{\U{211d} }%
%BeginExpansion
\mathbb{R}
%EndExpansion
^{dN}\rightarrow f\left( y,\xi \right) \in [0,+\infty) $ \\ 
$f\left(\cdot ,\xi\right) $ is Lebesgue measurable and $Y-$periodic for every $
\xi \in
\mathbb{R}^{dN}$ \\ 
$f\left( y,\cdot\right) $ is continous for a.e. $y\in 
\mathbb{R}
^{N}$%
\end{tabular}%
\end{equation}
\noindent and verifies the following assumptions: there exist $M>0$ and a $Y-$periodic function $a$ $\in L^{1}_{per}\left( Y\right) $
such that the following growth conditions hold
\begin{equation}\label{tch3}
f\left( y,\xi\right) \leq a\left( y\right) +MB\left( \left\vert \xi\right\vert
\right) \hbox{ for a.e. }y\in 
\mathbb{R}
^{N}\hbox{ and every }\xi\in
\mathbb{R}
^{dN},\, M >0 
\end{equation}%
\begin{equation}\label{tch4}
B\left( \left\vert \xi\right\vert \right) \leq f\left( y,\xi\right) \hbox{ for 
a.e. }y\in
\mathbb{R}
^{N}\hbox{ and every }\xi\in 
\mathbb{R}^{dN}.
\end{equation}

Here we are interested in the overall behaviour of such energies as the parameter $\varepsilon \to 0$, since it allows to replace a sample with finely distributed heterogeneities by a homogeneous material. 

Problems of this type arise in nonlinear elasticity theory, see \cite{ciora1} and the references therein. 
%Available result too, seem mainly to be
%in clasical Sobolev spaces \cite{ciora1}. We point out \cite{focardi 1,barba}, were the %study is done in the
%context of Orlicz setting.
While the theory is completely understood in the standard Sobolev setting, the analysis is not yet complete in the framework of Orlicz spaces, i.e. when the function $f$ satisfies  \eqref{tch3} and \eqref{tch4}.

Due to recent developments of two scale
convergence in the context of Orlicz setting \cite{fotso23,Elvira2,Elvira3} we are able to deduce the asymptotics of the functionals in \eqref{funct}, as $\varepsilon \to 0$, also in the Orlcz-Sobolev spaces. Indeed, our main results read as follows (we refer to Section \ref{secpre} for the adopted notation):

\begin{theorem}\label{main1}
	Let $B$ be a Young function satisfying $\nabla_2$ and $\Delta_2$  conditions, let  $\Omega \in \mathcal{A}_{0}$ and $u\in W^{1}L^B\left( \Omega ;
\mathbb{R}^{dN}\right) $. Let $\{\varepsilon\}$ be a family of positive numbers converging to $0$.
Assume that $f$ satisfies \eqref{tch1}-\eqref{tch4}, then for every sequence $\{\varepsilon_h\}\subset \{\varepsilon\}$
\begin{align*}
\inf \left\{ \underset{h\rightarrow \infty }{\lim \inf }\int_{\Omega
}f\left( \frac{x}{\varepsilon _{h}},\nabla u_{h}\right) dx:\left\{
u_{h}\right\} \subseteq W^{1}L^B\left( \Omega ;
\mathbb{R}
^{dN}\right) ,u_{h}\rightarrow u\text{ in }L^{B}\left( \Omega ;
\mathbb{R}
^{d}\right) \right\}  \\ 
=\inf \left\{ \underset{h\rightarrow \infty }{\lim \sup }\int_{\Omega
}f\left( \frac{x}{\varepsilon _{h}},\nabla u_{h}\right) dx:\left\{
u_{h}\right\} \subseteq W^{1}L^B\left( \Omega ;
\mathbb{R}
^{dN}\right) ,u_{h}\rightarrow u\text{ in }L^{B}\left( \Omega ;%
\mathbb{R}
^{d}\right) \right\}  \\ 
=\int_{\Omega }f_{\hom }\left( \nabla u\right) dx,
\end{align*}
\noindent where,  setting for every $t >0$, 
\begin{equation}\label{ft}
	f_{t}\left( \xi\right) :=\frac{1}{t^{N}}\inf \left\{
\int_{tY}f\left( y,\xi+\nabla v\right) dy:v\in W^1L^B_0\left( tY;
\mathbb{R}
^{d}\right) \right\}, 
\end{equation}
$f_{hom}:\xi\in 
\mathbb R
^{dN} \mapsto$
\begin{equation}\label{fhom}
\lim_{t \to +\infty} f_{t}\left( \xi\right) =
\frac{1}{t^N}\inf \left\{ \int_{tY}f\left( y,\xi+\nabla v\right) dy:v\in
W^1L^B_0\left( tY;
\mathbb{R}
^{d}\right) \right\}.
\end{equation}
\end{theorem}

We stress that our set of assumptions allows us to consider growth conditions modeled through $B(\xi):= |\xi|^p$, with $p>1$, hence Theorem \ref{main1} extends \cite[Theorem 2.5]{ciora1} to the case where no quasi-convexity hypothesis is made on $f(x,\cdot)$.

Our  main result follows by the extension to the Orlicz setting of the unfolding
method in homogenization (see \cite{fotso23} for the first results in this direction). Indeed, in subsection \ref{sec:con3}, we get

\begin{proposition}\label{maintwo}
Let $\Omega \in \mathcal{A}_{0}$  and let $B$ be a Young function satisfying $\nabla_2$ and $\Delta_2$ conditions. 
let $\left\{ \varepsilon _{h}\right\} $\ be a sequence of positive numbers
converging to $0.$ Let $\{v_{h}\}_h$ be a sequence weakly converging to some $v$ in $%
W^{1}L^B\left( \Omega \right) $. Then there exist a
subsequence  $h_{k}$ and $V$ $\in L^{B}\left( \Omega \times Y_{per}\right) $
such that  $\frac{\partial }{\partial y_{i}}V\left( x,\cdot \right) \in
L^{B}_{per}\left( Y\right) $ verifying, as $k\rightarrow \infty $ $\mathcal{T}%
_{\varepsilon _{h_{k}}}\left( \nabla v_{h_{k}}\right) \rightharpoonup \nabla
v+\nabla _{y}V$ in $L^{B}\left( \Omega \times Y_{per};%
\mathbb{R}
^{d}\right) .$
\end{proposition}

\color{black}
\bigskip The rest of the paper is organized as follows: section \ref{secpre} deals with
notation and preliminary results on Orlicz-Sobolev spaces, and the properties of the periodic  unfolding operator in Orlicz setting. Section \ref{secmain} is devoted to the proof of the main results. 
\color{black}
\section{Preliminary results}\label{secpre}

\subsection{Notations}
In this section, we start fixing some notation, which for the readers' convenience is not very different from the one adopted in \cite{FTNTZ}. We will also recall some preliminaries about Orlicz and Orlicz-Sobolev spaces that we will use in the sequel.
\begin{itemize}
	\item $\Omega$ is an open bounded subset of $\mathbb{R}^{N}$ with Lipschitz boundary.
	\item $\mathcal{A}(\Omega)$ denotes the family of all open subsets of $\Omega$.
	\item  $\mathcal{L}^{N}$ is the Lebesgue measure in $\mathbb{R}^{N}$.
	\item  $\mathbb{R}^{d\times N}$ is identified with the set of real $d\times N$ matrices.
	\item $Y := (0,1)^{N}$ is the unit cube in $\mathbb{R}^{N}$. 
	\item  The symbols $\langle\cdot\rangle$ and $[\cdot]$ stand, respectively, for the fractional and integer part of a number, or a vector, componentwise.
	\item  The Dirac mass at a point $a\in \mathbb{R}^{d}$ is denoted by $\delta_{a}$.
	\item The symbol $\dashint_{A}$ stands for the average $\mathcal{L}^{N}(A)^{-1}\int_{A}$.
	\item  $U$ is an open subset of $\mathbb{R}^{d}$.
	\item  $\mathcal{C}_{c}(U)$  is the space of continuous  functions $f: U\to \mathbb{R}$ with compact support.
	\item $B: [0,\infty)\to[0,\infty)$ is a Young function, i.e. $B$ is continuous, convex, with $B(t)>0$ for $t>0$, $\frac{B(t)}{t}\to 0$ as $t\to 0$, and $\frac{B(t)}{t}\to \infty$ as $t\to \infty$.
	\item $\widetilde{B}$ stands for the complementary of Young function $B$, defined by 
	\begin{equation*}
		\widetilde{B}(t)= \sup_{s\geq 0} \big\{st - B(s), \; t\geq 0 \big\}.
	\end{equation*} 
	\item We recall that a Young function $B$ is of class $\Delta_{2}$ at $\infty$ or satisfies the $\Delta_2$ condition (denoted $B\in \Delta_{2}$) if there are $\alpha>0$ and $t_{0}\geq 0$ such that 
	$$B(2t) \leq \alpha B(t),\; \textup{for\,all}\, t\geq t_{0}.$$
	Also $B$ is of class $\nabla_2$ (or satisfies $\nabla_2$ condition) if $\tilde B$ is of class $\Delta_2$, i.e. $\exists \beta >0$ and $t_0 >1$ such that  $$B(t) \leq \frac{1}{2 \beta} B(\beta t), \hbox{ for all }t \geq t_0.$$
	
	\item $L^{B}(\Omega;\mathbb{R}^{d})$ is the Orlicz space of functions defined by 
	\begin{equation*}
		L^{B}(\Omega;\mathbb{R}^{d}) = \left\{ u: \Omega\to \mathbb{R}^{d}\,;\, u\,\textup{is\,measurable},\; \lim_{\alpha\to 0}\int_{\Omega}B(\alpha|u(x)|)dx = 0 \right\}.
	\end{equation*}
	We recall that $L^{B}(\Omega;\mathbb{R}^{d})$ is a Banach space with respect to the Luxemburg norm 
	\begin{equation*}
		\|u\|_{B} = \inf \left\{ k>0\,:\, \int_{\Omega}B\left(\dfrac{|u(x)|}{k}\right)dx \leq 1 \right\}\, < \, +\infty.
	\end{equation*}
	Sometimes, we will denote the norm of elements in $L^{B}(\Omega;\mathbb{R}^{d})$, both by  $\|\cdot\|_{B}$ and with $\|\cdot\|_{L^{B}}$.
	\item   $\mathcal{D}(\Omega)$ is the space of  indefinitely differentiable functions $f: \Omega\to \mathbb{R}^{d}$ with compact support. We recall that $\mathcal{D}(\Omega)$ is dense in $L^{B}(\Omega;\mathbb{R}^{d})$, $L^{B}(\Omega;\mathbb{R}^{d})$ is separable and reflexive when $B$ satisfies $\Delta_{2}$ and $\nabla_2$ conditions above.  The dual of $L^{B}(\Omega;\mathbb{R}^{d})$ is identified with $L^{\widetilde{B}}(\Omega;\mathbb{R}^{d})$. The property $\frac{B(t)}{t}\to \infty$ as $t\to \infty$ implies that 
	\begin{equation*}
		L^{B}(\Omega;\mathbb{R}^{d}) \subset L^{1}(\Omega;\mathbb{R}^{d}) \subset L^{1}_{loc}(\Omega;\mathbb{R}^{d}) \subset \mathcal{D}'(\Omega),
	\end{equation*}
	each embedding being continuous.
	\item $W^{1}L^{B}(\Omega,\mathbb{R}^{d})$ is the Orlicz-Sobolev space defined by
	\begin{equation*}
		W^{1}L^{B}(\Omega,\mathbb{R}^{d}) = \left\{ u \in L^{B}(\Omega;\mathbb{R}^{d})\,;\, \dfrac{\partial u}{\partial x_{i}} \in L^{B}(\Omega;\mathbb{R}^{d}),\, 1\leq i\leq N \right\},
	\end{equation*}
	where derivatives are taken in the distributional sense on $\Omega$. Endowed with the norm 
	\begin{equation*}
		\|u\|_{W^{1}L^{B}} = \|u\|_{L^{B}} + \sum_{i=1}^{N} \left\|\dfrac{\partial u}{\partial x_{i}} \right\|_{L^{B}}, \; u \in W^{1}L^{B}(\Omega,\mathbb{R}^{d}),
	\end{equation*}
	$W^{1}L^{B}(\Omega,\mathbb{R}^{d})$ is a reflexive Banach space when $B \in \Delta_{2}\cap \nabla_2$. 
	\item $W^{1}_{0}L^{B}(\Omega,\mathbb{R}^{d})$ denotes the closure of $\mathcal{D}(\Omega)$ in $W^{1}L^{B}(\Omega,\mathbb{R}^{d})$ and the semi-norm 
	\begin{equation*}
		u \longrightarrow \|u\|_{W^{1}_{0}L^{B}} = \|Du\|_{L^{B}} =  \sum_{i=1}^{N} \left\|\dfrac{\partial u}{\partial x_{i}} \right\|_{L^{B}}
	\end{equation*}
	is a norm (of gradient) on $W^{1}_{0}L^{B}(\Omega,\mathbb{R}^{d})$ equivalent to $\|\cdot\|_{W^{1}L^{B}}$.
	%\item $Y=[0,1]^{N}$ is the unit cube of $\mathbb{R}^{N}_{y}$ (the $N$-dimensional numerical space $\mathbb{R}^{N}$ of variables $y=(y_{1},\cdots,y_{N})$). \\
	\item Given a function space $\mathcal{S}$ defined in $\Omega$, $Y$ or $\Omega\times Y$, the subscript $\mathcal{S}_{per}$ means that the functions are periodic in $\Omega$, $Y$ or $\Omega\times Y$, as it will be clear from the context.
	
	\item
	$L^B(\Omega \times \mathbb R^N_{loc})$ is the space defined by
	$$ L^B(\Omega \times \mathbb R^N_{loc}):=\left\{v \hbox{ measurable, s.t. } \int_{\Omega \times A} B(v(x,y))dxdy<+\infty \hbox{ for every }A \subset \subset \mathbb R^N_y\right\}.$$
	\item $L^{B}(\Omega\times Y_{per})$ is the Orlicz space defined by
\begin{equation*}
	L^{B}(\Omega\times Y_{per}) = \left\{ v \in L^B(\Omega \times \mathbb R^N_{loc}):  v(x,\cdot)\, \hbox{ is }\, Y-\textup{periodic} \hbox{ for a.e. } x \in \Omega\right\}.
\end{equation*}
%\item $L^{B}(\Omega;\mathcal{C}_{per}(Y))$ is the space of maps $\phi: \Omega\to \mathcal{C}_{per}(Y)$ such that 
%\begin{itemize}
%	\item[i)] $\phi$ is strongly measurable, i.e. there exists a sequence of simple functions $s_{n}: \Omega\to \mathcal{C}_{per}(Y)$ such $\|s_{n}(x)-\phi(x)\|_{\mathcal{C}_{per}(Y)} \to 0$ for a.e. $x\in\Omega$,
%	\item[ii)] $x\to \|\phi(x)\|_{\mathcal{C}_{per}(Y)}\in L^{B}(\Omega;\mathbb{R}^{d})$.
%\end{itemize}
%We recall that the embedding $L^{B}(\Omega;\mathcal{C}_{per}(Y)) \hookrightarrow L^{B}(\Omega\times Y_{per})$ is continuous with density.
\item $W^{1}L^{B}_{per}(Y)$ is the Orlicz-Sobolev space defined by
\begin{align*}
	W^{1}L^{B}_{per}(Y) = \big\{ v \in W^{1}L^{B}_{loc}(\mathbb{R}^{N}_{y}): v(x,\cdot)\,\hbox{ and } \tfrac{\partial v}{\partial y_{i}}(x,\cdot),\, 1\leq i\leq N,  \\
	\left. \hbox{ are }\, Y-\textup{periodic} \hbox{ for a.e. }x \in \Omega\right\}.
\end{align*}
	\item $W^{1}_{\#}L^{B}(Y)$ is the Orlicz-Sobolev space defined by
	\begin{equation*}
		W^{1}_{\#}L^{B}(Y) = \left\{ v \in W^{1}L^{B}_{per}(Y)\,:\, \dashint_{Y}v(y) dy = 0 \right\}.
	\end{equation*}
	It is endowed with the $L^B$ norm of the gradients.
	
	\item In our subsequent analysis we denote by %by $L^B(\Omega; L^B_{per}(Y))$ and $L^B(\Omega;L^B_{per}(Y \times Z))$ the spaces of functions in $L^B_{\rm loc}(\Omega \times Y)$ and $L^B_{\rm loc}(\Omega \times \mathbb R^d_y \times \mathbb R^d_z)$, respectively which are $Y$ and  $Y\times Z$ periodic
%	for a.e. $x \in \Omega$, respectively and whose Luxemburg  norm is finite in $\Omega \times K$, with $K$ being any compact set in  $Y$ and $Y \times Z$, respectively.
	\begin{align*}
		L^{B}\left( \Omega; L^B(Y_{per}) \right) :=\Big\{ u\in
		L^{B}\left( \Omega \times 
		\mathbb{R}
		^N_{loc}\right) : u\left( x,\cdot\right) \in
		L_{per}^{B}\left( Y\right)  \\ 
		\left. \hbox{for a.e. }x\in \Omega,\text{ and }\int_{\Omega \times Y}B\left( \left\vert u\left(
		x,y\right) \right\vert \right) dxdy<\infty \right\}.
	\end{align*}
%	\begin{align}\nonumber
%		L^{B}\left( \Omega;L^B_{per} (Y\times Z)\right) :=\Big\{ u\in
%		L_{loc}^{B}\left( \Omega \times 
%		\mathbb{R}
%		_{y}^d\times \mathbb{R}
%		_{z}^d\right) : u\left( x,\cdot,\cdot\right) \in
%		L_{per}^{B}\left( Y\times Z\right)  \\ 
%		\left. \hbox{for a.e. }x\in \Omega,\text{ and }\iiint_{\Omega \times Y\times Z}B\left( \left\vert u\left(
%		x,y,z\right) \right\vert \right) dxdydz<\infty \right\},\label{LBper}
%	\end{align}
%	respectively.  
We observe that our is a notation, indeed, in view of \cite[Lemma 2.4]{fotso nnang 2012}, if $\tilde B$ satisfies $\Delta'$ condition, then the above spaces coincide with the standard Orlicz-Bochner spaces, sometimes denoted in the same way.
\item Analogously we consider 
	\begin{align}\label{OBdef}
	L^{B}\left( \Omega; W^1_{\#} L^B(Y) \right) :=\Big\{ u\in
	L^{B}\left( \Omega \times 
	\mathbb{R}
	^N_{loc}\right) : u\left( x,y\right) \in
	L^B(\Omega;L_{per}^{B}\left( Y\right)) , u(x,\cdot) \in W^{1}_{\#} L^B(Y) \nonumber \\ 
	\left. \hbox{for a.e. }x\in \Omega,\hbox{ and }\int_{\Omega \times Y}\left( B(|u(x,y)|)+ B\left( \left\vert \nabla u\left(
	x,y\right) \right\vert \right)\right) dxdy<\infty \right\}. 
\end{align}

	%\item Throughout the paper the letter $E$ will denote any ordinary sequence $E=(\varepsilon_{n})$ (integers $n\geq 0$) with $0 \leq \varepsilon_{n} \leq 1$ and $\varepsilon_{n} \rightarrow 0$ as $n\rightarrow \infty$. Such a sequence will termed a \textit{fundamental sequence}. Also, we will assume that the Young functions $B, \widetilde{B}$ are of class $\Delta_{2}$ at $\infty$.
\end{itemize}

\subsection{The unfolding operator}\label{sec:con3}

This subsection is devoted to recall the unfolding operator introduced in \cite{CDG2} and some of its properties in the Orlicz setting proven in \cite{fotso23}.  

\noindent Here and in the sequel  $\Omega \in \mathcal A_0(\mathbb R^N)$ %to be a bounded open subset with Lipschitz boundary of $\mathbb R^N$
and $Y:=]0,1[^N$. %\color{blue} and, unless otherwise specified, the positive constant $C$ may vary from line to line. \color{black} 
Following \cite{CDG2}, for $z\in 
%TCIMACRO{\U{211d} }%
%BeginExpansion
\mathbb{R}^N ,\left[ z\right] _{Y}$ is the vector with components $\left[ z_{i}\right] 
$ where $\left[ z_{i}\right] $\ is the integer part of $z_{i}.$ It follows
that $z-\left[ z\right] _{Y}=\left\{ z\right\} _{Y}\in Y.$ Then for each $%
x\in 
\mathbb{R}^N$, $x=\varepsilon\left( \left[ \frac{x}{\varepsilon }\right] _{Y}+\left\{ \frac{x}{%
	\varepsilon }\right\} _{Y}\right) .$ 

Define \begin{align}\label{2.1CDG2}
	\Xi_{\varepsilon }:=\left\{ \xi \in 
	\mathbb{Z}
	^{d},\;\varepsilon \left( \xi +Y\right) \subset \Omega \right\},\nonumber\\ 
	\widehat{\Omega }_{\varepsilon }:= {\rm int} \left\{ \bigcup\limits_{\xi \in \Xi
		_{\varepsilon }}\varepsilon \left( \xi +\overline{Y}\right) \right\},\\
	\Lambda _{\varepsilon }:=\Omega \backslash \widehat{\Omega }_{\varepsilon }.\nonumber
\end{align}
The set $\widehat \Omega_\varepsilon$  is the largest union of cells $\varepsilon(\xi + \overline Y)$ (with $\xi \in \mathbb Z^N$) included in $\Omega$ while $\Lambda_\varepsilon$ is
the subset of $\Omega$ containing the parts from cells $\varepsilon(\xi +\overline Y)$ intersecting the boundary $\partial \Omega$.

The unfolding operator, introduced in \cite[Definition 2.1]{CDG2}, which acts on Lebesgue measurable functions is defined as follows.

\begin{definition}\label{unfdef}
	For $\phi $ Lebesgue measurable on $\Omega$, the unfolding operator $%
	\mathcal{T}_{\varepsilon }$ is defined as 
	\begin{equation*}
		\mathcal{T}_{\varepsilon }\left( \phi \right) \left( x,y\right) =\left\{ 
		\begin{tabular}{ll}
			$\phi \left( \varepsilon \left[ \frac{x}{\varepsilon }\right]
			_{Y}+\varepsilon y\right) $ & a.e for $\left( x,y\right) \in \widehat{\Omega 
			}_{\varepsilon }\times Y$ \\ 
			$0$ & a.e for $\left( x,y\right) \in \Lambda _{\varepsilon }\times Y.$%
		\end{tabular}%
		\right.
	\end{equation*}%
	Clearly $\mathcal T_\varepsilon (\phi)$ is Lebesgue measurable in $\Omega \times Y$, and is $0$ whenever $x$ is outside $\widehat \Omega_\varepsilon$. Moreover, for every $v,w $ Lebesgue-measurable, $$\mathcal{T}_{\varepsilon }\left( vw\right) =\mathcal{T}%
	_{\varepsilon }\left( v\right) \mathcal{T}_{\varepsilon }\left( w\right)$$
\end{definition}

Analogous identities hold for other pointwise operations on functions. In particular, the operator is linear with respect to pointwise operations. In the sequel we recall a series of result whose proofs can be found in \cite{fotso23}.

\begin{proposition}
	Let $f\in L_{per}^{1}\left( Y\right) $ define the sequence $\left\{
	f_{\varepsilon }\right\}_\varepsilon $ by $f_{\varepsilon }\left( x\right) :=f\left( 
	\frac{x}{\varepsilon }\right) $ a.e for $x\in 
	%TCIMACRO{\U{211d} }%
	%BeginExpansion
	\mathbb{R}
	%EndExpansion
	^{n},$ then 
	\begin{equation*}
		\mathcal{T}_{\varepsilon }\left( f_{\varepsilon }|_{\Omega }\right) \left(
		x,y\right) =\left\{ 
		\begin{tabular}{ll}
			$f\left( y\right) $ & \;a.e for $\left( x,y\right) \in \widehat{\Omega }%
			_{\varepsilon }\times Y$, \\ 
			$0$ & \;a.e for $\left( x,y\right) \in \Lambda _{\varepsilon }\times Y.$%
		\end{tabular}%
		\right.
	\end{equation*}%
	If $f\in L_{per}^{B}\left( Y\right)$, %,and if $\Omega $\ is bounded, 
	$\mathcal{T}_{\varepsilon }\left( f_{\varepsilon }|_{\Omega }\right)
	\to f$ strongly in $L^{B}\left( \Omega \times Y\right) .$
\end{proposition}

 The following result can be immediately deduced by \cite[Proposition 2.5]{CDG2}, observing that for every Young function $B$ and $w \in L^B(\Omega)$ $B(\mathcal T_\varepsilon(w))= \mathcal T_\varepsilon B(w)$.
\color{black}

\begin{proposition}\label{Prop2}
	For every Young function $B$, the operator $\mathcal{T}_{\varepsilon }$\ is linear and
	continuous from $L^{B}\left( \Omega \right) $ to $L^{B}\left( \Omega \times
	Y\right) $. 
	It results that
	\begin{itemize}
		\item[i)] $\frac{1}{|Y|}\int_{\Omega \times Y}
		B(\mathcal T_\varepsilon (w))(x, y) dx dy =\int_\Omega B(w(x))dx- \int_{\Lambda_\varepsilon}B(w(x))dx= \int_{\widehat \Omega_\varepsilon} B(w(x))dx.$
		\\	\item[ii)]$ \frac{1}{|Y|}\int_{\Omega \times Y}
		B(\mathcal T_\varepsilon (w))(x, y) dx dy \leq \int_\Omega B(w(x))dx$
		\\	\item[iii)] $\frac{1}{|Y|}|\int_{\Omega \times Y}
		B(\mathcal T_\varepsilon (w))(x, y) dx dy -\int_\Omega B(w(x))dx|\leq \int_{\Lambda_\varepsilon}B(w(x))dx$
		\\	\item[iv)]$\|\mathcal T_\varepsilon(w)\|_{L^B(\Omega \times Y)} = \|w \chi_{\widehat \Omega_\varepsilon}\|_{L^B(\Omega)}$,
	\end{itemize}	
	with $\chi_{\widehat \Omega_\varepsilon}=\left\{\begin{array}{ll} 
		1 & \;\hbox{ in }\widehat \Omega_\varepsilon,\\
		0 & \;\hbox{otherwise},
	\end{array}\right.$

	In particular, for every $\varepsilon, $
	\begin{align}\label{est1}	\left\Vert \mathcal{T}%
		_{\varepsilon }\left( w\right) \right\Vert _{L^{B}\left( \Omega \times
			Y\right) }\leq \left( 1+\left\vert Y\right\vert \right) \left\Vert
		w\right\Vert _{L^{B}\left( \Omega \right) }.
	\end{align}
\end{proposition}

From this result, in particular from  ii), it is possible to provide, as in the standard $L^p$ setting (cf. \cite[Proposition 2.6]{CDG2}), an unfolding criterion for integrals in the Orlicz setting.
For the sake of a more complete parallel with the standard $L^p$ setting, we recall the  unfolding criterion for integrals,  u.c.i. for short, as introduced in \cite{CDG2}. 

\begin{proposition}\label{u.c.i.}
	If $\left\{ w _{\varepsilon }\right\}_\varepsilon $\ is a
	sequence in $L^{1}\left( \Omega \right) $ satisfying $\int_{\Lambda
		_{\varepsilon }}\left| w _{\varepsilon }\right| dx\to
	0$\ as $\varepsilon \to 0,$ then \begin{align*}\ \int_{\Omega }
		w_{\varepsilon } dx- \frac{1}{\left\vert Y\right\vert }%
		\int_{\Omega \times Y}\mathcal{T}_{\varepsilon }\left( w _{\varepsilon
		}\right) dxdy\to 0\end{align*} as $\varepsilon \to 0$.
\end{proposition}
This result justifies the following notation for integrals of unfolding operators. Indeed, if $\left\{ w_{\varepsilon }\right\}_\varepsilon$ is a sequence
satisfying $u.c.i$, we write $$\int_{\Omega } w_{\varepsilon
} dx\overset{\mathcal{T}_{\varepsilon }}{\simeq }\frac{1}{%
	\left\vert Y\right\vert }\int_{\Omega \times Y}\mathcal{T}_{\varepsilon
}\left( w_{\varepsilon }\right) dxdy.$$

\begin{proposition}[\bf u.c.i. in the Orlicz setting]\label{uciO}
	If $\left\{ w _{\varepsilon }\right\}_\varepsilon $\ is a
	sequence in $L^B\left( \Omega \right) $ satisfying $\int_{\Lambda
		_{\varepsilon }}B( w _{\varepsilon })dx\to
	0$\ as $\varepsilon \to 0,$ then $$\ \int_{\Omega }\
	B(w_{\varepsilon }) dx-\frac{1}{\left\vert Y\right\vert }%
	\int_{\Omega \times Y}B\left(\mathcal{T}_{\varepsilon }\left( w _{\varepsilon
	}\right)\right) dxdy\to 0 $$\ as $\varepsilon \to 0$.
\end{proposition}
\begin{proposition}\label{Prop5}
	Let $\left\{ u_{\varepsilon }\right\}_\varepsilon $ be a bounded sequence in $
	L^{B}\left( \Omega \right) $ and $v\in L^{\widetilde{B}}\left( \Omega
	\right) $ then $$\int_{\Omega }u_{\varepsilon }vdx\overset{\mathcal{T}_{\varepsilon }}{\simeq } \frac{1}{\left\vert Y\right\vert }\int_{\Omega
		\times Y}\mathcal{T}_{\varepsilon }\left( u_{\varepsilon }\right) \mathcal{T}%
	_{\varepsilon }\left( v\right) dxdy.$$
\end{proposition}

As in the classical Lebesgue setting we can define the mean value operator acting on $L^B$ spaces.

\begin{definition}\label{meandef}
	The mean value operator $\mathcal{M}_{Y}\colon L^{B}\left( \Omega
	\times Y\right) \to L^{B}\left( \Omega \right) $ is defined as
	follows $$\mathcal{M}_{Y}\left( w \right) \left( x\right) :=\frac{
		1}{\left\vert Y\right\vert }\int_Y w\left( x,y\right) dy$$ for
	a.e. $x\in \Omega$ and for every $w \in L^B(\Omega \times Y)$.
\end{definition}

In particular $\left\Vert \mathcal{M}_{Y}\left( w\right) \right\Vert
	_{L^{B}\left( \Omega \right) }\leq \frac{1}{\sup\{1,|Y|^{-1}\}}%left\vert Y\right\vert ^{-1}
	\left\Vert
	w \right\Vert _{L^{B}\left( \Omega \times Y\right) },$ for every $w
	\in L^{B}\left( \Omega \times Y\right) .$ 
	
We also recall
the convergence properties related to the unfolding operator
when $\varepsilon \to 0$, i.e.  for $w$ uniformly continuous on $\Omega$, with modulus of continuity $m_w$, it is
easy to see that
\begin{align*}
	\sup_{x\in \widehat{\Omega}_\varepsilon,y \in Y}
	|{\mathcal T}_\varepsilon(w)(x, y) -w(x)| \leq m_w(\varepsilon).
\end{align*}

The following theorem extends to the Orlicz setting the correspective one in classical $L^p$ spaces, (cf. \cite[Proposition 2.9]{CDG2} and \cite[Theorem 1]{fotso23}).

\begin{theorem}\label{convBunf}
	Let $B$ be a Young function satisfying $\nabla_2$ and  $\Delta_2$ conditions. The following results hold:
	\begin{itemize}
		\item[i)] For $w\in L^{B}( \Omega) ,\mathcal{T}_{\varepsilon }\left( w\right) \to w$ strongly in $L^{B}\left(
		\Omega \times Y\right) $.
		
		\item[ii)] Let $\left\{ w_{\varepsilon }\right\}_\varepsilon $ be a sequence in 
		$L^{B}\left( \Omega \right) $ such that $w_{\varepsilon }\to w$
		strongly in $L^{B}\left( \Omega \right) ,$ then $\mathcal{T}_{\varepsilon
		}\left( w_\varepsilon\right) \to w$ strongly in $L^{B}\left( \Omega \times
		Y\right) .$
		
		\item[iii)] \, For every relatively weakly compact sequence $\left\{
		w_{\varepsilon }\right\}_\varepsilon $ in $L^{B}\left( \Omega \right) $\ the
		corresponding $\left\{ \mathcal{T}_{\varepsilon }\left( w_{\varepsilon
		}\right) \right\}_\varepsilon $ is relatively weakly compact in $L^{B}\left( \Omega
		\times Y\right) .$
		Furthermore,
		if $\mathcal{T}_{\varepsilon }\left( w_{\varepsilon }\right)
		\rightharpoonup \widehat{w}$ weakly in $L^{B}\left( \Omega \times Y\right) ,$
		then $w_{\varepsilon }\rightharpoonup \mathcal{M}_{Y}\left( \widehat{w}%
		\right) $ weakly in $L^{B}\left( \Omega \right) $.
		
		\item[iv)] If $\mathcal{T}_{\varepsilon }\left( w_{\varepsilon
		}\right) \rightharpoonup \widehat{w}$ weakly in $L^{B}\left( \Omega \times
		Y\right) ,$ then 
		\begin{align*}%\label{est2}
			\left\Vert \widehat{w}\right\Vert _{L^{B}\left( \Omega
				\times Y\right) }\leq \underset{\varepsilon \rightarrow 0}{\lim \inf }\left(
			1+\left\vert Y\right\vert \right) \left\Vert w_{\varepsilon }\right\Vert
			_{L^{B}\left( \Omega \right) }.
		\end{align*}
\end{itemize}
\end{theorem}

\subsection{Two scale convergence and unfolding}

We start recalling the notion of two-scale convergence in Orlicz setting, see \cite[Definition 4.1]{fotso nnang 2012}.
%To this end we recall
%\begin{align}\label{LBOmegaYper}
%	L^B (\Omega \times Y_{per}) := \{u \in L^B_{loc}
%	(\Omega \times \mathbb R^N_y): u(x, \cdot) \hbox{ is } Y -	\hbox{periodic, for a.e. } x \in \Omega \}
%\end{align}
%taking into account that $L^B(\Omega, C_{per}(Y))$  is continuously embedded dense subset of  $L^B(\Omega \times Y_{per})$.

\begin{definition}\label{twoscaleconv}
	A sequence $\{v_\varepsilon\}_\varepsilon \subset L^B(\Omega)$ is said to
 (weakly) two-scale converge in $L^B(\Omega)$ to some $v_0 \in L^B (\Omega \times Y_{per})$ if
		have 
		\begin{equation*}%\label{2seq}
			\int_\Omega v_\varepsilon \varphi^\varepsilon dx \to
			\int_{\Omega \times Y} v_0 \varphi dx dy
		\end{equation*}
		for every $\varphi \in L^{\tilde B}(\Omega; C_{per} (Y))$, where $\varphi^\varepsilon$  is defined as $\varphi^\varepsilon(x):= \varphi\left(x, \frac{x}{\varepsilon}\right)$, for every $x \in \Omega$.
	%	\item[(ii)] strongly two-scale converge in $L^B(\Omega)$ to some $v_0 \in L^B (\Omega\times Y_{per})$ if the following
	%	property is verified:
	%	Given $\eta > 0$ and $\varphi \in L^{B}(\Omega; C_{per} (Y))$ with $\|v_0 - \varphi\|_{L^B(\Omega \times Y)} \leq \eta$, there
	%	exists some $\alpha > 0$ such that $\|v_\varepsilon - \varphi^\varepsilon\|_{L^B(\Omega)} \leq \eta$ provided $0 < \varepsilon \leq \alpha$.
	%\end{itemize}
\end{definition}
We express the above definition by saying that $v_\varepsilon \to v_0$ in $L^B(\Omega)$ weakly two-scale. %and the case $(ii)$ by $v_\varepsilon \to v_0$ in
%$L^B(\Omega)$-strongly two-scale.

The following result, which we restate for the readers' convenience, is the Orlicz version of \cite[Proposition 1.19]{CDG3} and is a consequence of it (see Proposition \ref{maintwo}).

\begin{proposition}Let $B$ be an Young function satisfying $\nabla_2$ and $\Delta_2$ conditions. Let $\left\{ w_{\varepsilon }\right\}_\varepsilon$ be a bounded sequence in $%
	L^{B}(\Omega).$ The following assertions are equivalent:
	
	$\left( i\right) \left\{ \mathcal{T}_{\varepsilon }\left( w_{\varepsilon
	}\right) \right\}_\varepsilon $ converges weakly to $w$ in $L^{B}\left( \Omega \times
	Y\right)$.
	
	$\left( ii\right) \left\{ w_{\varepsilon }\right\}_\varepsilon$ weakly two-scale converges to $	w$ in $L^B(\Omega)$.
\end{proposition}

\begin{proof}
	The proof relies on \cite[Proposition 3.5]{CM} and \cite[Proposition 1.19]{CDG3}. First we recall that each Young function $B$ which satisfies $\nabla_2$ and $\Delta_2$ condition is such that there exists $p, q>1$ $L^p \subseteq L^B \subseteq L^q$ with continuous embeddings.
	
	First we assume that  (i) holds. The above embeddings ensure that $\{\mathcal T_\varepsilon(w_\varepsilon)\}_\varepsilon$  is also bounded in $L^q(\Omega \times Y)$, hence it converges, up to a subsequence, to the same limit $w$ in $L^q(\Omega \times Y)$. By the uniqueness of the limit, the entire sequence converges. By \cite[Proposition 1.19]{CDG3}, it results that $\{w_\varepsilon\}_\varepsilon$ is weakly two-scale convergent to $w$ in $L^q$, but by (iii) in Theorem \ref{convBunf}, the sequence is also bounded in $L^B$, hence by \cite[Theorem 4.1]{fotso nnang 2012}, two-scale weakly converging, up to a subsequence also in $L^B$. By the uniqueness of the weak two-scale limit, this limiting function must be  $w$, and the convergence has to hold for the entire sequence. 
	
	Analogously, assuming that $(ii)$ holds $\{w_\varepsilon\}_\varepsilon$ is bounded in $L^q(\Omega)$, thus, it is also weakly two-scale convergent also in $L^q(\Omega \times Y)$ to the same $w$. By  \cite[Proposition 1.19]{CDG3} $\{\mathcal T_\varepsilon(w_\varepsilon)\}_\varepsilon$  converges weakly to $w$ in $L^q(\Omega \times Y)$. On the other hand, by (iii) in Theorem \ref{convBunf}, the sequence $\{\mathcal T_\varepsilon(w_\varepsilon)\}_\varepsilon$ is also bounded in $L^B(\Omega \times Y)$, hence weakly convergent (up to a subsequence) in $L^{B}(\Omega\times Y)$ to the same limit, which in the end is the limit of the entire sequence.  
\end{proof}
%\color{black}
%
%We emphasize that the above arguments exploited the fact the uniqueness of the limit and the metrizability of weak convergence in bounded sets.

%\par\egroup

\color{black}
\noindent The following result has been proven in \cite[Theorem 4.2]{fotso nnang 2012} (see also \cite[Remark 2]{FTGNZ})
\begin{proposition}\cite{fotso nnang 2012}
	Let $\{u_{\varepsilon}\}_{\varepsilon} \subset W^{1}L^{B}(\Omega)$ be a bounded sequence. Then, there exists a not relabelled subsequence $\{u_{\varepsilon}\}_{\varepsilon}$, such that
	\begin{equation*}
	u_{\varepsilon} \rightharpoonup u_{0} \quad \textup{in} \; W^{1}L^{B}(\Omega)-weak,
	\end{equation*}
	\begin{equation*}
	u_{\varepsilon} \rightharpoonup u_{0} \quad \textup{in} \; L^{B}(\Omega)-\hbox{weakly two scale},
	\end{equation*}
	\begin{equation*}
	\dfrac{\partial u_{\varepsilon}}{\partial x_{i}}	 \rightharpoonup \dfrac{\partial u_{0}}{\partial x_{i}} + \dfrac{\partial u_{1}}{\partial y_{i}} \quad  \;L^{B}(\Omega) \hbox{ weakly two scale }, (1 \leq i \leq N),
	\end{equation*}
	where $u_{0} \in W^{1}L^{B}(\Omega)$ and $u_{1} \in L^{B}(\Omega ; W^{1}_{\#}L^{B}_{per}(Y))$, (where the latter is in the sense of \eqref{OBdef}). %\color{magenta} with $\frac{\partial u_{1}}{\partial y_{i}} \in L^{B}(\Omega\times Y_{per})$.Is it needed? \color{black}
	Furthermore 
	if $\{u_{\varepsilon}\}_{\varepsilon} \subset  W^{1}_{0}L^{B}(\Omega)$ then the weak limit $u_{0}$ lies in $W^{1}_{0}L^{B}(\Omega)$,
	
\end{proposition}

\subsection{Properties of the energy densities}

We recall that a function $f:\mathbb R^d\to [0,+\infty)$ is said to be  quasiconvex if for every $\xi \in \mathbb R^{dN}$,

\begin{equation*}%\label{qcx}
	f\left(\xi\right)\leq \frac{1}{\mathcal L^N(A)}\int_A f(\xi +\nabla \varphi(x))dx
	\end{equation*}
	for every $\varphi \in W^{1,\infty}(A;\mathbb R^d)$.
	A function $f:\Omega \times \mathbb R^{dN}\to [0,+\infty]$ is said to be quasiconvex if
	\begin{equation}\label{tch2}
		f(y,\cdot) \hbox{ is quasiconvex for a.e. } y\in 
	\Omega.
\end{equation}

It is well knonw (see \cite{daco}) that every quasiconvex 
function $f:\mathbb R^d \to \mathbb R$ is separately convex. Furtermore, if $f:\Omega \times \mathbb R^{dN}\to [0,+\infty)$ satisfies \eqref{tch3} for a Young function $B$ of class $\nabla_2\cap \Delta_2$, in \cite{focardi 1} it has been proven that
	\begin{equation}\label{prop2.8}
	\left\vert f\left( x,z_{1}\right) -f\left( x,z_{2}\right) \right\vert \leq
	C\left( 1+b\left( 1+\left\vert z_{1}\right\vert +\left\vert z_{1}\right\vert
	\right) \right) \left\vert z_{1-}z_{2}\right\vert, 
	\end{equation}
with the constant $C$ independent on $x$, and $b:[0,+\infty)\to [0,+\infty)$ is nondecreasing, right continuous and such that
\begin{align*}
	B(t)=\int_0^t b(s)ds,\\
	b(0)=0, b(s)>0 \; s>0, \lim_{s \to +\infty} b(s)=+\infty. 
\end{align*}

The following result has been proved in \cite[Lemma 2.1]{ciora1}

\begin{lemma}\label{ar2}\cite{ciora1} Let $f:\Omega \times \mathbb R^{dN}\to [0,+\infty)$ be such that $f(\cdot, \xi) \in L^1(Y)$ for every $\xi \in \mathbb R^{dN}$,  and assume that it satisfs \eqref{tch1}. Let  $t\in \mathbb R^+$ and let $f_{t}:\mathbb R^{dN}\to [0,+\infty)$ be as in \eqref{ft}, then
 for all \bigskip $\xi\in 
	%TCIMACRO{\U{211d} }%
	%BeginExpansion
	\mathbb{R}
	%EndExpansion
	^{dN},$ $\lim_{t \to +\infty}f_t(\xi)$ exists and 
	\begin{equation*}
		\underset{t\rightarrow +\infty }{\lim }f_{t}\left( \xi\right) =\underset{h\in 
			%TCIMACRO{\U{2115} }%
			%BeginExpansion
			\mathbb{N}
			%EndExpansion
		}{\inf }f\left( \xi\right). 
	\end{equation*}
\end{lemma}

\subsection{Lower semicontinuity and relaxation results}\label{lscrel}
The following result can be easily deduced by \cite[Theorem 4.7 and Theorem 4.11]{MM}

\begin{proposition}\label{relMM}
	Let $\Omega $ be a bounded open subset of $\mathbb R^N$ and let $f:\Omega \times \mathbb R^{dN}\to [0,+\infty)$ be a Carath\'eodory function such that \eqref{tch3} and \eqref{tch4} (with $a \in L^1_{\rm loc}(\mathbb R^N)$) hold. Let $F: W^{1}L^B\left( \Omega ;%
	%TCIMACRO{\U{211d} }%
	%BeginExpansion
	\mathbb{R}
	%EndExpansion
	^{d}\right) \to \mathbb R$
	 be the functional defined as %
	 \begin{equation*}%\label{F(w)}
F(w):= \int_{\Omega }f\left( x,\nabla w\left( x\right) \right)
	dx,
	\end{equation*}
	and let
	$\overline F:w \in W^{1}L^B(\Omega;\mathbb R^d)$ be the functionl defined as
	\begin{equation*}%\label{relFdef}
		\overline F(u):=\inf\{\liminf_{h \to +\infty} F(u_h): W^{1}L^{B}(\Omega;\mathbb R^d) \ni u_h \rightharpoonup u \hbox{ in } W^1L^B(\Omega;\mathbb R^d)\}
		\end{equation*}
	then  
	\begin{align*}%\label{Frel}
		\overline F(u)=\int_\Omega Qf(x,\nabla u(x))dx,
		\end{align*}
	where $Qf(x,\cdot)$ stands for the quasiconvex envelope of $f(x,\cdot)$ in the sense of Morrey, for a.e. $x \in \Omega$,
	i.e.
	$Q f: \Omega \times \mathbb R^{dN} \to \mathbb R$ is such that
	\begin{equation}\label{Qcxf}
	Q f(x,\xi)=\frac{1}{{\mathcal L^N}(A)}\inf\left\{\int_A Qf(x, \xi+ \nabla \varphi(y))d y: \varphi \in C^\infty_0(A)\right\},\end{equation}
	for a.e. $x \in \Omega$ and every $\xi \in \mathbb R^{dN}$,
	equivalently in view of \eqref{tch3} and \eqref{tch4} (see \cite{Dac},\cite{wlo1})
	\begin{equation*}%\label{Qcxfeq}
		Q f(x,\xi)=\frac{1}{{\mathcal L^N}(A)}\inf\left\{\int_A Qf(x, \xi+ \nabla \varphi(y))d y: \varphi \in W^{1}L^B_0(A)\right\},\end{equation*}
\end{proposition}
\begin{remark}
	\label{remqcxf}
	We observe that \cite[proof of Proposition 4.9 and Theorem 4.11]{MM} ensure that
	$Qf(x,\xi)$ is a Carath\'eodory function satisfying \eqref{tch3} and \eqref{tch4}, which is periodic in the first variable in view of  \eqref{Qcxf}.
\end{remark}

The following result is well known, hence its proof is omitted.
 
\begin{lemma}\label{lem1} Let $\Omega \in \mathcal A_0$ and let $f:\Omega \times \mathbb R^{dN}\to [0,+\infty)$, such that \eqref{tch1}, \eqref{tch3} and \eqref{tch4} hold. Let $\left\{ \nu
	_{h}\right\} \subseteq \left] 0,+\infty \right[ ,$ nondecreasing and
	diverging then for all  $w\in W^{1}L^B\left( \Omega ;%
	\mathbb{R}
	^d\right) ,$   
	\begin{equation*}
	\begin{array}{l}
	\inf \left\{ \underset{h\rightarrow \infty }{\lim \inf }\int_{\Omega
	}f\left( \nu _{h}x,\nabla u_{h}\right) dx:\left\{ u_{h}\right\} \subseteq
	W^{1}L^B\left( \Omega ;
	\mathbb{R}
	^{dN}\right) ,u_{h}\rightarrow u\text{ in }L^{B}\left( \Omega ;%
	\mathbb{R}
	^{d}\right) \right\} \\ 
	 =\inf \left\{ \underset{h\rightarrow \infty }{\lim \inf 
	}\int_{\Omega }f\left( \nu _{h}x,\nabla u_{h}\right) dx:\left\{
	u_{h}\right\} \subseteq u+W^{1}L^B_0\left( \Omega ;%
	\mathbb{R}
	%EndExpansion
	^{dN}\right) ,u_{h}\rightarrow u\text{ in }L^{B}\left( \Omega ;%
	\mathbb{R}
	^{d}\right) \right\} .
	\end{array}
	\end{equation*}
\end{lemma}

\section{Proof of main result}\label{secmain}

We start this subsection with a crucial observation. In view of  \cite[Proposition 6.11]{DM}, and Remark \ref{remqcxf}, $Qf(x,\xi)$ satisfies \eqref{tch1}-\eqref{tch4} and \eqref{tch2}.
Furthermore, it is easily seen from Theorem \ref{main1} and \eqref{fhom} that the formula defining $f_{hom}$ does not change if one replaces $f(x,\cdot)$ by $Qf(x,\cdot)$.
Thus, without loss of generality in this subsection we can assume that $f$ satisfies \eqref{tch2}.
%\begin{equation}\label{fqcx}
%	f(x,\cdot) \hbox{ is
 %quasiconvex for a.e. }x \in \Omega.\end{equation}
It is worth to underline that our set of assumptions allows us to consider growth conditions modeled through $B(\xi):= |\xi|^p$, with $p>1$, hence Theorem \ref{main1} extends \cite[Theorem 2.5]{ciora1}.

\bigskip 

The following result, whose proof is presented for the readers' convenience, is an Orlicz-Sobolev version of \cite[Lemma 2.2]{ciora1}.

\begin{lemma}\label{ar1}
Let $\Omega \in \mathcal{A}%
_{0}$, $t>0$. Assume that $f$ satisfy \eqref{tch1}-\eqref{tch4} and \eqref{tch2} and let $f_{t}$ be as in \eqref{ft}, then 
\begin{align*}
\int_{\Omega }f_{t}\left( \nabla u\right) dx
=\frac{1}{t^{n}}\inf \left\{ \int_{\Omega \times tY}f\left( y,\nabla
u\left( x\right) +\nabla _{y}V\left( x,y\right) \right) dxdy:V\in
L_{D_{0y}}^{B}\left( \Omega \times tY;\mathbb R^d\right) \right\}  
\end{align*}
for all $u\in W^{1}L^B(\Omega ;\mathbb R^d) $,
where
\begin{align*}L_{D_{0y}}^{B}( \Omega \times tY;\mathbb R^d) = 
\left\{ V\in L^{B}( \Omega \times tY;\mathbb R^d) :V\left( x,\cdot\right) \in W^1L^{B}_0\left( tY;\mathbb R^d\right) \hbox{ for a.e. }x\in \Omega \right\} .
\end{align*}
\end{lemma}

\begin{proof}
For every $V \in L_{D_{0y}}^{B}( \Omega \times tY;\mathbb R^d)$, it results that $\nabla_y V(x,y)\in L^{B}\left( \Omega \times tY;%
%TCIMACRO{\U{211d} }%
%BeginExpansion
\mathbb{R}
%EndExpansion
^{d}\right) $ for a.e $x\in \Omega,$
hence, by \eqref{ft}, we have
\begin{align*}
f_{t}\left( \nabla u\right) = &\frac{1}{t^{n}}\inf \left\{ \int_{tY}f\left( y,\nabla u\left( x\right)
	+\nabla v\left( y\right) \right) dy;v\in W^1L^B_0\left( tY;\mathbb R^d\right) \right\} \\
\leq &\frac{1}{t^{n}}\int_{tY}f\left( y,\nabla u\left( x\right) +\nabla
_{y}V\left( x,y\right) \right) dy 
\end{align*}
which leads to 
\begin{align*}
\int_{\Omega }f_{t}\left( \nabla u\right) dx
\leq\frac{1}{t^{n}}\inf \left\{ \int_{\Omega \times tY}f\left( y,\nabla u\left(
x\right) +\nabla _{y}V\left( x,y\right) \right) dxdy:V\in
L_{D_{0y}}^{B}\left( \Omega \times tY;
\mathbb R ^d\right) \right\}  
\end{align*}
To prove the reverse inequality, we assume, without loss of generality, that the left hand side of above inequality
is finite. Set $X=W^1L^B_0\left( tY;\mathbb R^d\right) $. Note that it is a metric space and define the multifunction 
 \begin{align*}\Gamma :\xi\in 
\mathbb{R}
^{Nd}\mapsto \left\{ v\in W^1L^B_0\left( tY;\mathbb R^d\right) :\frac{1}{t^{n}}\int_{tY}f\left( y,\xi+\nabla _{y}v\left( y\right)
\right) dy=f_{t}\left( \xi\right) \right\} .
\end{align*} By \eqref{tch4} and the lower semicontinuity properties ensured by Proposition \ref{relMM}, $\Gamma \left( \xi\right) $ is not empty and weakly closed. Let $F$ be a strongly closed set in $W^1L^B_0\left( tY;%
\mathbb{R}^d\right) $ hence $\Gamma \left( F\right) ^{-}=\left\{ \varsigma \in 
\mathbb{R}
^{dN}:\Gamma \left( \varsigma \right) \cap F\neq \emptyset \right\} \in 
\mathcal{B}\left(
\mathbb{R}
^{dN}\right) .$
Indeed, if $\Gamma \left( F\right) ^{-}=\emptyset $, clearly it is in $%
\mathcal{B}\left(
\mathbb{R}
^{dN}\right) .$

\noindent

The proof follows the same lines as \cite[Lemma 2.2]{ciora1}, by an application of Castaing's selection lemma (see \cite[Theorem 1.1]{ciora1}).
In particular, by the separability of $W^1L^B_0(ty;\mathbb R^d)$, it suffices to restrict to the case in which $F$ is closed ball and prove that $\Gamma \left( F\right) ^{-}$
is closed. Indeed every closed subset of $W^1L^B_0(ty;\mathbb R^d)$ can be written as countable intersection of closed balls and consewuently $\Gamma^-(F)$ will be a closed intersection of closed sets.

 To see that $\Gamma \left( F\right) ^{-}$
 is closed, let $\left\{ \xi_{s}\right\} \subseteq \Gamma \left(
F\right) ^{-}$ converging to $\xi$ and $v_{s}\in \Gamma \left( \xi_{s}\right)
\cap F.$ $F$ being a closed ball, $\left\{ v_{s}\right\} $ is bounded in $X$
hence there is a subsequence $\left\{ v_{s_{k}}\right\} $\ of $\left\{
v_{s}\right\} $ and  $v_{\infty }\in F$ (note that $F$ as ball is
convex and weakly closed) such that $v_{s_{k}}%
\rightharpoonup v_{\infty }$ in $X.$ 

\noindent We deduce, using definition of infimum and the continuity properties, that:
 \begin{align*}f_{t}\left( \xi\right) &\leq 
\frac{1}{t^{n}}\int_{tY}f\left( y,\xi+\nabla _{y}v_{\infty }\left( y\right)
\right) dy 
\leq \underset{s\rightarrow \infty }{\lim \inf }\frac{1}{t^{n}}
\int_{tY}f\left( y,\xi_{s}+\nabla_{y}v_{s_{k}}\left( y\right) \right) dy \\
&\leq \underset{s\rightarrow \infty }{\lim \sup }f_{t}\left(
\xi_{s_{k}}\right) \leq f_{t}\left( \xi\right),
\end{align*} 
that is $v_{\infty }\in \Gamma
\left( z\right) \cap F$ and $z\in \Gamma \left( F\right) ^{-}$ hence it is
closed then in $\mathcal{B}\left( 
\mathbb{R}
^{dN}\right).$
It follows, by Castaing's selection theorem, that $\Gamma $\ has
a measurable selection $\sigma .$ For a.e. $x\in \Omega $ we set $U\left(
x\right) =\sigma \left( \nabla u\left( x\right) \right).$ In particular, $U$ is 
Lebesgue measurable on $\Omega $ with  
\begin{align*}f_{t}\left( \nabla u\left( x\right)
\right) =\frac{1}{t^{n}}\int_{tY}f\left( y,\nabla u\left( x\right) +\nabla
_{y}U\left( x,y\right) \right) dy.
\end{align*} Since $f_{t}\left( \nabla u\right) $ is
summable, the previous equality, together with (\ref{tch1}), implies $%
\nabla _{y}U\in L^{B}\left( \Omega \times tY;\mathbb R
^{dN}\right) .$
Consequently $U\in L_{D_{0y}}^{B}\left( \Omega \times tY;
\mathbb{R}
^{dN}\right) $ with,%
\begin{align*}
&\int_{\Omega }f_{t}\left( \nabla u\left( x\right) \right) dx  =\frac{1}{t^{n}}\int_{\Omega \times tY}f\left( y,\nabla u\left( x\right)
+\nabla _{y}U\left( x,y\right) \right) dxdy \\ 
&\geq \frac{1}{t^{n}}\inf \left\{ \int_{\Omega \times tY}f\left( y,\nabla
u\left( x\right) +\nabla _{y}V\left( x,y\right) \right) dxdy:V\in
L_{D_{0y}}^{B}\left( \Omega \times tY;
\mathbb{R}
^{d}\right) \right\} ,
\end{align*}
ending the proof.
\end{proof}

As in \cite{ciora1} from lemmas \ref{ar1}, \ref{ar2},  (\ref{tch3})
and the Lebesgue dominated convergence theorem, the following result holds.

\begin{proposition}\label{ar3}
Let  $f$ be satisfying \ref{tch1}-\eqref{tch4}, let $f_{hom}$ be as in \eqref{fhom} and let $B$ be a Young function $B$ of class $\Delta_{2},$ let $\Omega \in 
\mathcal{A}_{0}$ and $u\in W^{1}L^B\left( \Omega ; \mathbb R^d
\right) $, them
\begin{align*}
\exists \lim_{t \to \infty} \frac{1}{t^{n}}\inf \left\{
\int_{\Omega \times tY}f\left( y,\nabla u\left( x\right) +\nabla _{y}V\left(
x,y\right) \right) dxdy:V\in L_{D_{0y}}^{B}\left( \Omega \times tY;%
\mathbb{R}
^{d}\right) \right\}  \\ 
=\int_{\Omega }f_{\hom }\left( \nabla u\left( x\right) \right) dx
\end{align*}
\end{proposition}

\begin{remark}
Under the assumptions of Proposition \ref{ar3} the following equivalent formulalae for the integral of $f_{hom}$ hold:
 \begin{align*}
%\begin{tabular}{l}
\int_{\Omega }f_{\hom }\left( \nabla u\left( x\right) \right) dx \\ 
=\underset{t\rightarrow \infty }{\lim }\frac{1}{t^{n}}\inf \left\{
\int_{\Omega \times tY}f\left( y,\nabla u\left( x\right) +\nabla _{y}V\left(
x,y\right) \right) dxdy:V\in L_{D_{0y}}^{B}\left( \Omega \times tY;\mathbb R^d
\right) \right\}  \\ 
=\underset{t\rightarrow \infty }{\lim }\inf \left\{ \int_{\Omega \times
Y}f\left( ty,\nabla u\left( x\right) +\nabla _{y}V\left( x,y\right) \right)
dxdy:V\in L_{D_{0y}}^{B}\left( \Omega \times Y;\mathbb R^d\right) \right\} \\ 
=\underset{h\in 
\mathbb{N}
}{\inf }\inf \left\{ \int_{\Omega \times tY}f\left( hy,\nabla u\left(
x\right) +\nabla _{y}V\left( x,y\right) \right) dxdy:V\in
L_{D_{0y}}^{B}\left( \Omega \times Y;\mathbb R^d\right) \right\} .
\end{align*}
\end{remark}

For the readers'convenience we restate our main homogenization result

\begin{theorem}\label{ta1}
	Let $B$ be a Young function satisfying $\nabla_2$ and $\Delta_2$  conditions, let  $\Omega \in \mathcal{A}_{0}$ and $u\in W^{1}L^B\left( \Omega ;
\mathbb{R}^{dN}\right) $. Let $\{\varepsilon\}$ be a family of positive numbers converging to $0$.
Assume that $f$ satisfies \eqref{tch1}-\eqref{tch4}, then for every sequence $\{\varepsilon_h\}\subset \{\varepsilon\}$, Then:%
\begin{align*}
&\int_{\Omega }f_{\hom }\left( \nabla u\right) dx\\ 
&=\inf \left\{ \underset{h\rightarrow \infty }{\lim \inf }\int_{\Omega
}f\left( \frac{x}{\varepsilon _{h}},\nabla u_{h}\right) dx:\left\{
u_{h}\right\} \subseteq W^{1}L^B\left( \Omega; \mathbb R^{dN}\right) ,u_{h}\rightarrow u\text{ in }L^{B}\left( \Omega ;\mathbb R^{d}\right) \right\}  \\ 
&=\inf \left\{ \underset{h\rightarrow \infty }{\lim \sup }\int_{\Omega
}f\left( \frac{x}{\varepsilon _{h}},\nabla u_{h}\right) dx:\left\{
u_{h}\right\} \subseteq W^{1}L^B\left( \Omega ;
\mathbb R^{dN}\right) ,u_{h}\rightarrow u\text{ in }L^{B}\left( \Omega ;
\mathbb R^{d}\right) \right\}  \\ 
&=\inf \left\{ \underset{h\rightarrow \infty }{\lim \inf }\int_{\Omega
}f\left( \frac{x}{\varepsilon _{h}},\nabla u_{h}\right) dx:\left\{
u_{h}\right\} \subseteq u+W^1L^B_0\left( \Omega ;\mathbb R^{dN}\right) ,u_{h}\rightarrow u\text{ in }L^{B}\left( \Omega ;\mathbb R^{d}\right) \right\}  \\ 
&=\inf \left\{ \underset{h\rightarrow \infty }{\lim \sup }\int_{\Omega
}f\left( \frac{x}{\varepsilon _{h}},\nabla u_{h}\right) dx:\left\{
u_{h}\right\} \subseteq u+W^1L^B_0\left( \Omega ;
\mathbb R^{dN}\right) ,u_{h}\rightarrow u\text{ in }L^{B}\left( \Omega ;%
\mathbb R^{d}\right) \right\}.
\end{align*}%
The proof is split in many intermediate steps as follow.
\end{theorem}

Firstly we observe that the exact same arguments used in the proof of \cite[Lemma 2.6]{ ciora1} guarantee the validity of the following result
\begin{lemma}\label{lem3.2}
Under the same assumptions of Theorem \ref{ta1} and recalling that \eqref{tch2} holds, the inequality%
\begin{align*}
\inf \left\{ \underset{h\rightarrow \infty }{\lim \sup }\int_{\Omega
}f\left( \frac{x}{\varepsilon _{h}},\nabla u_{h}\right) dx:\left\{
u_{h}\right\} \subseteq W^{1}L^B\left( \Omega ;\mathbb R^{dN}\right) ,u_{h}\rightarrow u\text{ in }L^{B}\left( \Omega ;\mathbb R^{d}\right) \right\} \\ 
\leq \underset{k\in \mathbb{N}
}{\inf }\frac{1}{k^{n}}\inf \left\{ \int_{\Omega \times kY}f\left( y,\nabla
u\left( x\right) +\nabla _{y}V\left( x,y\right) \right) dxdy,\right.\\
\left.V\in L^{B}( \Omega \times tY;\mathbb R^d): V(x,\cdot) \in
W^1L^{B}_{per}\left( \Omega \times kY;
\mathbb{R}
^{d}\right) \hbox{ for a.e. }x \in \Omega\right\},
\end{align*}
follows.
\end{lemma}

%Denoting by \begin{align*}L_{D_{per y}}^{B}( \Omega \times tY;\mathbb R^d) = 
%	\left\{ V\in L^{B}( \Omega \times tY;\mathbb R^d) :V\left( x,\cdot\right) \in %W^1L^{B}_{per}\left( tY;\mathbb R^d\right) \hbox{ for a.e. }x\in \Omega \right\} .
%\end{align*}
%\color{black}

We observe that since $W^1L^B_0\left( tY;
\mathbb{R}
^{d}\right) \subseteq W^{1}L^{B}_{per}\left( tY;
\mathbb{R}
^{d}\right) $ the lemma is valid for also for $V\in L_{D_{0y}}^{B}\left( \Omega
\times kY;
\mathbb{R}
^{d}\right) .$

%\begin{lemma}
%Satisfies (\ref{tch1}) ,(\ref{tch2}) ,(\ref{tch3}) and (\ref{tch4}) for an N-function $B$ of class $\triangle _{2}.$ For $%
%t>0,$
%\end{lemma}

The exact same arguments as in \cite{ciora1}, allows us to deduce the following result, which is a consequence of  lemmas  \ref{ar1}, \ref{ar2}, \ref{ar3} and the dominated convergence theorem, taking into account that \eqref{tch3} entails that $f_t$ in \eqref{ft} satisties the following
\begin{equation*}%\label{ftgrowt}
f_t(\xi)\leq \int_Y a(ty) dy + B(|\xi|),\end{equation*}
for every $t>0$ and $\xi \in \mathbb R^{dN}$.

\begin{proposition}\label{prop2.3}
	Let $f$ satisfy \eqref{tch1}-\eqref{tch4} and \eqref{tch3} for a given Young function $B$ satisfying $\nabla_2$ and $\Delta_2$ conditions. 
Then, for every $\Omega \in \mathcal{A}_{0}$ and $u\in W^{1}L^B\left( \Omega ;\mathbb{R}
^{d}\right) $ we have:%
\begin{align*}
\underset{t\rightarrow \infty }{\lim }\frac{1}{t^{n}}\inf \left\{
\int_{\Omega \times tY}f\left( y,\nabla u\left( x\right) +\nabla _{y}V\left(
x,y\right) \right) dy:V\in L_{D_{0y}}^{B}\left( \Omega \times tY;
\mathbb{R}^d\right) \right\}  \\ 
=\int_{\Omega }f_{\hom }\left( \nabla u\right) dx. \\ 
%L_{D_{0y}}^{B}\left( \Omega \times tY;
%\mathbb{R}^d\right) =\left\{ u\in L^{B}\left( \Omega \times tY;
%\mathbb{R}^{d}\right) :\nabla _{y}V\left( x,y\right) \in L^{B}\left( \Omega \times tY;%
%\mathbb{R}
%^{d}\right) \right\}.
\end{align*}
\end{proposition}

%\begin{lemma}
%(lemma 2.6) \ Asume $f$\ satisfies (\ref{tch1}) and (\ref{tch3}) for an N-function $B.$ Let $\left( \varepsilon _{h}\right)
%\subseteq \left] 0,+\infty \right[ $ converge to $0,\Omega \in \mathcal{A}%
%_{0}$ and $u\in W^{1}L^B\left( \Omega ;%
%%TCIMACRO{\U{211d} }%
%%BeginExpansion
%\mathbb{R}
%%EndExpansion
%^{d}\right) .$ Then 
%\begin{equation*}
%\begin{tabular}{l}
%$\underset{t\rightarrow \infty }{\lim }\frac{1}{t^{n}}\inf \left\{
%\int_{\Omega \times tY}f\left( y,\nabla u\left( x\right) +\nabla _{y}V\left(
%x,y\right) \right) dy:V\in L_{D_{0y}}^{B}\left( \Omega \times tY;%
%%TCIMACRO{\U{211d} }%
%%BeginExpansion
%\mathbb{R}
%%EndExpansion
%^{d}\right) \right\} $ \\ 
%$\inf \left\{ \underset{h\rightarrow 0}{\lim \sup }\int_{\Omega \times
%tY}f\left( y,\nabla u\left( x\right) +\nabla _{y}V\left( x,y\right) \right)
%dy:V\in L_{D_{0y}}^{B}\left( \Omega \times tY;%
%%TCIMACRO{\U{211d} }%
%%BeginExpansion
%\mathbb{R}
%%EndExpansion
%^{d}\right) \right\} $ \\ 
%$=\int_{\Omega }f_{\hom }\left( \nabla u\right) dx.$ \\ 
%$L_{D_{0y}}^{B}\left( \Omega \times tY;%
%%TCIMACRO{\U{211d} }%
%%BeginExpansion
%\mathbb{R}
%%EndExpansion
%^{d}\right) =\left\{ u\in L^{B}\left( \Omega \times tY;%
%%TCIMACRO{\U{211d} }%
%%BeginExpansion
%\mathbb{R}
%%EndExpansion
%^{d}\right) :\nabla _{y}V\left( x,y\right) \in L^{B}\left( \Omega \times tY;%
%%TCIMACRO{\U{211d} }%
%%BeginExpansion
%\mathbb{R}
%%EndExpansion
%^{d}\right) \right\} $%
%\end{tabular}%
%\end{equation*}
%\end{lemma}

\begin{lemma}\label{lemma2.8}
Let $f$ satisfy \eqref{tch1}-\eqref{tch4} and \eqref{tch3} for a given Young function $B$ of class $\triangle _{2}\cap \nabla_2$. 
Then, for every $\Omega \in \mathcal{A}_{0}$ and $u\in W^{1}L^B\left( \Omega ;\mathbb{R}
^{d}\right) $, and for every $\varepsilon_h \to 0$, we have: 
\begin{align}\label{eqtoest}
\underset{\nu \in 
	\mathbb{N}
}{\sup }\inf \left\{ \underset{h\rightarrow \infty }{\lim \inf }\int_{\Omega
}f\left( \nu hx,\nabla u+\nabla v_{h}\right) dx:\left\{ v_{h}\right\}
\subseteq W^1L^B_0\left( \Omega ;
\mathbb{R}
^{d}\right) ,v_{h}\rightarrow 0\text{ in }L^{B}\left( \Omega;
\mathbb{R}
^{d}\right) \right\}\\
 \leq 
\inf \left\{ \underset{h\rightarrow \infty }{\lim \inf }\int_{\Omega
}f\left( \frac{x}{\varepsilon _{h}},\nabla u+\nabla u_{h}\right) dx:\left\{
u_{h}\right\} \subseteq W^1L^B_0\left( \Omega ;
\mathbb{R}
^{d}\right) ,u_{h}\rightarrow 0\text{ in }L^{B}\left( \Omega ;%
\mathbb{R}
^{d}\right) \right\} \nonumber
\end{align}
\end{lemma}

\begin{proof}
Let $\left\{ u_{h}\right\} \subseteq W^1L^B_0\left( \Omega ;%
\mathbb{R}^{d}\right) ,u_{h}\rightarrow 0$ in $L^{B}\left( \Omega ;%
\mathbb{R}
^{d}\right).$ 
Without loss of generality assume that $\underset{h\rightarrow \infty }{\lim \inf }%
\int_{\Omega }f\left( \frac{x}{\varepsilon _{h}},\nabla u+\nabla
u_{h}\right) dx<+\infty .$ Choose $\Omega ^{\prime }\in \mathcal{A}_{0}$ such
that $\overline{\Omega }\subseteq \Omega ^{\prime }$, and, with an abuse of notation denote by $u_{h}$ and $u$ both the functions and their extensions as $0$ in $\Omega'$. By \eqref{tch3}, it follows that:%
\begin{equation}\label{eq3}
\begin{tabular}{l}
$\int_{\Omega ^{\prime }}f\left( \frac{x}{\varepsilon _{h}},\nabla u+\nabla
u_{h}\right) dx$ \\ 
$\leq \int_{\Omega }f\left( \frac{x}{\varepsilon _{h}},\nabla u+\nabla
u_{h}\right) dx+\int_{\Omega ^{\prime }\setminus \Omega }\left( a\left( \frac{x}{\varepsilon _{h}}\right)%
+MB\left( \left\vert \nabla u_{h}\right\vert \right) \right) dx,\forall h\in 
%TCIMACRO{\U{2115} }%
%BeginExpansion
\mathbb{N}
%EndExpansion
.$%
\end{tabular}%
\end{equation}%
Let $\{h_j\}\subset \mathbb N$ such that 
 $\left\{ \left[ \frac{1}{\nu \varepsilon _{h_{j}}}\right] \right\} $ is
strictly nondecresing and
\begin{equation}\label{eq2}
\underset{j\rightarrow \infty }{\lim }
\int_{\Omega }f\left( \frac{x}{\varepsilon _{h}},\nabla u+\nabla
u_{h_{j}}\right) dx=\underset{h\rightarrow \infty }{\lim \inf }\int_{\Omega
}f\left( \frac{x}{\varepsilon _{h}},\nabla u+\nabla u_{h}\right) dx<+\infty.
\end{equation}
 %Hence $\left\{ \int_{\Omega }f\left( \frac{x}{\varepsilon _{h}},\nabla
%u+\nabla u_{h_{j}}\right) dx\right\} $ is bounded 

Therefore \eqref{tch4} ensures that $\underset{j}{\sup }\int_{\Omega }B\left( \left\vert
\nabla u_{h_{j}}\right\vert \right) dx<\infty $, and, equivalently, that $\underset{j}{%
\sup }\left\Vert \nabla u_{h_{j}}\right\Vert _{B,\Omega }<\infty $  (recall that $%
\left\Vert \nabla u_{h_{j}}\right\Vert _{B,\Omega }\leq \int_{\Omega
}B\left( \left\vert \nabla u_{h_{j}}\right\vert \right) dx$ for $\left\Vert
\nabla u_{h_{j}}\right\Vert _{B,\Omega }\geq 1).$  For $\nu $ and $j$
nonnegative integers, let $\theta _{j}=\nu \varepsilon _{h_{j}}\left[ \frac{1%
}{\nu \varepsilon _{h_{j}}}\right] $ and set $v_{j}=\frac{1}{\theta _{j}}%
u_{h_{j}}\left( \theta _{j} \cdot\right) .$ Observing that $0\leq
\theta _{j}\leq 1,$ with $\underset{j\rightarrow \infty }{\lim }\theta
_{j}=1, $ and that $v_{j}\left( x\right) =0,\forall x\notin \frac{1}{\theta
_{j}}\overline{\Omega }$ and,  taking into account that, for $j$ large enough $\frac{1}{\theta _{j}}%
\overline{\Omega }\subseteq \Omega ^{\prime },$ one has that $v_{j}\in
W^1L^B_0\left( \Omega ^{\prime };%
%TCIMACRO{\U{211d} }%
%BeginExpansion
\mathbb{R}
%EndExpansion
^{d}\right) .$ Moreover, since $u_{h}\rightarrow
0 $ in $L^{B}\left( \Omega ;%
%TCIMACRO{\U{211d} }%
%BeginExpansion
\mathbb{R}
%EndExpansion
^{d}\right) ,$ also $v_{j}\rightarrow 0$ in $L^{B}\left( \Omega ;%
%TCIMACRO{\U{211d} }%
%BeginExpansion
\mathbb{R}
%EndExpansion
^{d}\right), $ and $\underset{j}{\sup }\int_{\Omega
^{\prime }}B\left( \left\vert \nabla v_{_{j}}\right\vert \right) dx<\infty $, 
i.e. $\underset{j}{\sup }\left\Vert \nabla v_{_{j}}\right\Vert
_{B,\Omega ^{\prime }}<\infty .$ 
Moreover, as in \cite{ciora1}, the change of variables
 $x=\theta
_{j}x^{\prime },$for $j$ sufficiently large, ensures that

\begin{align}
\int_{\Omega }f\left( \frac{x}{\varepsilon _{h}},\nabla u+\nabla
u_{h_{j}}\right) dx=\theta _{j}^{n}\int_{\frac{1}{\theta _{j}}\Omega
^{\prime }}f\left( \frac{\theta _{j}x^{\prime }}{\varepsilon _{h_{j}}}%
,\nabla u+\nabla u_{h_{j}}\right) \left( \theta _{j}x^{\prime }\right)
dx^{\prime }\nonumber \\ 
=\theta _{j}^{n}\int_{\frac{1}{\theta _{j}}\Omega ^{\prime }}f\left( \frac{%
\theta _{j}x^{\prime }}{\varepsilon _{h_{j}}},\nabla u\left( \theta
_{j}x^{\prime }\right) +\nabla _{x^{\prime }}v_{_{j}}\left( x^{\prime
}\right) \right) dx^{\prime } \label{toquote} \\ 
\geq \theta _{j}^{n}\int_{\Omega }f\left( \frac{\theta _{j}x^{\prime }}{%
\varepsilon _{h_{j}}},\nabla u\left( \theta _{j}x^{\prime }\right) +\nabla
_{x^{\prime }}v_{_{j}}\left( x^{\prime }\right) \right) dx^{\prime }. \nonumber
\end{align}%
%Therefore 
%\begin{equation*}
%\begin{tabular}{l}
%$\int_{\Omega ^{\prime }}f\left( \frac{x}{\varepsilon _{h}},\nabla u+\nabla
%u_{h_{j}}\right) dx\geq \theta _{j}^{n}\int_{\Omega }f\left( \frac{\theta
%_{j}x^{\prime }}{\varepsilon _{h_{j}}},\nabla u\left( x^{\prime }\right)
%+\nabla v_{_{j}}\left( x^{\prime }\right) \right) dx^{\prime }$ \\ 
%$-\theta _{j}^{n}\int_{\frac{1}{\theta _{j}}\Omega ^{\prime }}f\left( \frac{%
%\theta _{j}x^{\prime }}{\varepsilon _{h_{j}}},\nabla u\left( \theta
%_{j}x^{\prime }\right) +\nabla _{x^{\prime }}v_{_{j}}\left( x^{\prime
%}\right) \right) dx^{\prime }$ \\ 
%$\geq \theta _{j}^{n}\int_{\Omega }f\left( \frac{\theta _{j}x^{\prime }}{%
%\varepsilon _{h_{j}}},\nabla u\left( \theta _{j}x^{\prime }\right) +\nabla
%_{x^{\prime }}v_{_{j}}\left( x^{\prime }\right) \right) dx^{\prime }.$%
%\end{tabular}%
%\end{equation*}%

By \eqref{prop2.8},  it results that, for all $\xi_{1},\xi_{2}\in 
\mathbb{R}
%EndExpansion
^{dN},\varepsilon >0;$ 
\begin{align*}
f\left( \xi_{1}\right) -f\left( \xi_{2}\right) \leq 
%\left\vert f\left(
%\xi_{1}\right) -f\left( \xi_{2}\right) \right\vert \leq \left( 1+ b\left(
%1+\left\vert \xi_{1}\right\vert +\left\vert \xi_{2}\right\vert \right) \right)
%\left\vert \xi_{1}-\xi_{2}\right\vert $ \\ 
\varepsilon \left( 1+b\left( 1+\left\vert \xi_{1}\right\vert +\left\vert
\xi_{2}\right\vert \right) \right) \frac{\left\vert \xi_{1}-\xi_{2}\right\vert }{%
\varepsilon } \\
\leq B\left( \frac{\left\vert \xi_{1}-\xi_{2}\right\vert }{\varepsilon }\right)
+\varepsilon \widetilde{B}\left( 1+b\left( 1+\left\vert \xi_{1}\right\vert
+\left\vert \xi_{2}\right\vert \right) \right)  \\ 
\leq C_{1\varepsilon }B\left( \left\vert \xi_{1}-\xi_{2}\right\vert \right)
+\varepsilon C_{2}\widetilde{B}\left( a\right) +\varepsilon C_{2}\widetilde{B%
}\left( b\left( 1+\left\vert \xi_{1}\right\vert +\left\vert \xi_{2}\right\vert
\right) \right)\\ 
\leq C_{1\varepsilon }B\left( \left\vert \xi_{1}-\xi_{2}\right\vert \right)
+\varepsilon C_{2}\widetilde{B}\left( a\right) +\varepsilon C_{3}B\left(
1+\left\vert \xi_{1}\right\vert +\left\vert \xi_{2}\right\vert \right)  \\ 
\leq C_{1\varepsilon }B\left( \left\vert \xi_{1}-\xi_{2}\right\vert \right)
+\varepsilon C_{2}\widetilde{B}\left( a\right) +\varepsilon C_{3}C_{4}\left(
B\left( 1\right) +B\left( \left\vert \xi_{1}\right\vert \right) +B\left(
\left\vert \xi_{2}\right\vert \right) \right) , 
\end{align*}
for suitable nonnegative constants $C_1,C_2,C_3$ and $C_4$, and taking into account that
in the second and forth lines it has been used the duality relation between the two convex conjugate functions $B$ and $\tilde B$, in the third and the fifth ones it has been exploited the fact that both $B$ and $\widetilde{B}$ satisfy $\Delta_2$.

That is, \eqref{toquote} can be estimated as follows
%\begin{equation*}
%\begin{tabular}{l}
%$f\left( z_{2}\right) \geq $ \\ 
%$f\left( z_{1}\right) -C_{1\varepsilon }B\left( \left\vert
%z_{1}-z_{2}\right\vert \right) -\varepsilon C_{2}\widetilde{B}\left(
%a\right) -\varepsilon C_{3}C_{4}\left( B\left( 1\right) -B\left( \left\vert
%z_{1}\right\vert \right) -B\left( \left\vert z_{2}\right\vert \right)
%\right) $ \\ 
%$C_{1\varepsilon },C_{2},C_{3},C_{4}$ are constant.%
%\end{tabular}%
%\end{equation*}%
%we deduce: 
\begin{align}
\int_{\Omega ^{\prime }}f\left( \frac{x}{\varepsilon _{h}},\nabla u+\nabla
u_{h_{j}}\right) dx\geq \theta _{j}^{n}\int_{\Omega }f\left( \frac{\theta
_{j}x^{\prime }}{\varepsilon _{h_{j}}},\nabla _{x}u\left( x^{\prime }\right)
+\nabla _{x^{\prime }}v_{_{j}}\left( x^{\prime }\right) \right) dx^{\prime }\nonumber
\\ 
-\theta _{j}^{n}C_{1\varepsilon }\int_{\Omega }B\left( \left\vert \nabla
_{x}u\left( \theta _{j}x^{\prime }\right) -\nabla _{x}u\left( x^{\prime
}\right) \right\vert \right) dx-\varepsilon \int_{\Omega }C_{2}\widetilde{B}%
\left( 1\right) dx-\varepsilon C_{3}C_{4}\int_{\Omega }\left( B\left(
1\right) dx\right) \label{eq4} \\ 
-\varepsilon C_{3}C_{4}\int_{\Omega }B\left( \left\vert \nabla u+\nabla
u_{h_{j}}\right\vert \right) dx-\varepsilon C_{3}C_{4}\int_{\Omega }B\left(
\left\vert \nabla _{x}u\left( x^{\prime }\right) +\nabla _{x^{\prime
}}v_{_{j}}\left( x^{\prime }\right) \right\vert \right) dx,
\end{align} for $j$ large
enough and for $n_{j}=\left[ \frac{1}{\nu \varepsilon _{h_{j}}}\right].$ Hence taking the limit (up to a subsequence), as $j \to +\infty$ in \eqref{eq4},
in view of the arbitrariness of $\varepsilon$ and the bounds, we obtain

\begin{align}
\underset{j\rightarrow \infty }{\lim }\int_{\Omega }f\left( \nu
n_{j}x,\nabla u+\nabla v_{j}\right) dx\underset{j\rightarrow \infty }{=\lim
\inf }\int_{\Omega }f\left( \nu n_{j}x,\nabla u+\nabla v_{j}\right) dx \nonumber \\ 
=\underset{j\rightarrow \infty }{\lim \inf }\theta _{j}^{n}\int_{\Omega
}f\left( \nu n_{j}x,\nabla u+\nabla v_{j}\right) dx \label{eq5} \\ 
=\underset{j\rightarrow \infty }{\lim \inf }\theta _{j}^{n}\int_{\Omega
}f\left( \frac{\theta _{j}x}{\varepsilon _{h_{j}}}x,\nabla u+\nabla
v_{j}\right) dx \nonumber \\ 
\leq \underset{j\rightarrow \infty }{\lim }\int_{\Omega ^{\prime }}f\left( 
\frac{x}{\varepsilon _{h_{j}}}x,\nabla u+\nabla u_{h}\right) dx. \nonumber
\end{align}

Then, from \eqref{eq5}, \eqref{eq2}, \eqref{eq3}, the properties of $a$ in \eqref{tch3}, we get
\begin{align*}%\label{eq6}
\underset{j\rightarrow \infty }{\lim \inf }\int_{\Omega }f\left( \nu
n_{j}x,\nabla u+\nabla v_{j}\right) dx \\ 
\underset{j\rightarrow \infty }{\leq \lim \inf }\int_{\Omega }f\left( \frac{%
x}{\varepsilon_{h_j}}x,\nabla u+\nabla u_{h}\right) dx+\mathcal L^N( \Omega
^{\prime }\backslash \Omega) \int_Y a(x)dx
+M\int_{\Omega ^{\prime }\backslash \Omega }B\left( \left\vert \nabla
u\right\vert \right) dx. \nonumber
\end{align*}%
Set, for every $h \in \mathbb N$ 
\begin{equation*}
w_{h}=\left\{ 
\begin{tabular}{l}
$v_{j}$ if $h=n_{j}$ for some $j\in 
\mathbb{N}
$ \\ 
$0$ otherwise,
\end{tabular}%
\right.
\end{equation*}%
Thus, $\left\{ w_{h}\right\} \subseteq W^{1}L^B\left( \Omega ;%
\mathbb{R}
^{d}\right) ,$ $w_{h}\rightarrow 0$ in $L^{B}\left( \Omega ;
\mathbb{R}
^{d}\right) .$ Letting $\Omega' \ssearrow \Omega $, and taking in account that $\mathcal L^N(\partial \Omega) =0$ we get 
\begin{align*}
\inf \big\{ \underset{h\rightarrow +\infty }{\lim \inf }\int_{\Omega
}f\left( \nu hx,\nabla u\left( x\right) +\nabla v_{h}\right) dx: \left\{
v_{h}\right\} \subseteq W^{1}L^B\left( \Omega ;%
\mathbb{R}
^{Nd}\right), \\
v_{h}\rightarrow 0\text{ in }L^{B}\left( \Omega ;
\mathbb{R}
^{d}\right) \big\}  \\ 
\leq \underset{h\rightarrow +\infty }{\lim \inf }\int_{\Omega }f\left( \nu
hx,\nabla u\left( x\right) +\nabla w_{h}\right) dx \\ \leq \underset{%
j\rightarrow +\infty }{\lim \inf }\int_{\Omega }f\left( \nu n_{j}x,\nabla
u\left( x\right) +\nabla v_{j}\right) dx \\ 
\leq \underset{j\rightarrow +\infty }{\lim \inf }\int_{\Omega }f\left( \frac{x}{%
\varepsilon _{h}}x,\nabla u\left( x\right) +\nabla u_{h}\right) dx.
\end{align*}%

\noindent Finally, the conclusion follows by lemma \ref{lem1} and by taking the supremum over $\nu $.
\end{proof}

\begin{lemma}\label{lem3.6}
Let $f$ satisfy \eqref{tch1}-\eqref{tch4} and \eqref{tch3} for a given Young function $B$ satisfying $\nabla_2$ and $\Delta_2$ condition. 
Then, for every $\Omega \in \mathcal{A}_{0}$ and $u\in W^{1}L^B\left( \Omega ;\mathbb{R}^{d}\right) $, we have:
\begin{align*}
\inf \left\{ 
\underset{h\rightarrow +\infty }{\lim \inf }%
\int_{\Omega }f\left( \frac{x}{\varepsilon_{h}},\nabla u\left( x\right)
+\nabla u_{h}(x)\right) dx:\left\{ u_{h}\right\} \subseteq W^1L^B_0\left(
\Omega;\mathbb R^{d}\right),\right.\\ 
u_{h}\rightarrow 0\text{ in }L^{B}\left( \Omega ;
\mathbb{R}
^{d}\right)
\big\}  \\ 
\geq \underset{h\rightarrow +\infty }{\lim \inf }\inf \left\{ \int_{\Omega
\times Y}f\left( hy,\nabla u\left( x\right) +\nabla _{y}V\left( x,y\right)
\right) dxdy:V\in L_{D_{0y}}^{B}\left( \Omega \times Y;
\mathbb{R}
^{d}\right) \right\}.
\end{align*}
\end{lemma}

\begin{proof}
	To prove the lemma, we estimate from below the right-hand side of the inequality
	in Lemma \ref{lemma2.8}
	The regularity of $\Omega $ allows us to extend  $u$ and any sequence $\{v_h\}\subseteq  W^1L^B_0\left( \Omega ;
	\mathbb{R}
	^{d}\right) $ converging to $0$ from $W^{1}L^B\left( \Omega ;
\mathbb{R}
^{d}\right) $\ to $W^{1}L^B\left( 
\mathbb{R}
^{N};
\mathbb{R}
^{d}\right),$ with an abuse of notation, the extension of $\{v_h\}$ outside $\Omega$ will be still denoted by $\{v_h\}$. Let $\nu \in
\mathbb{N},
$ %and $v_{h}\subseteq  W^1L^B_0\left( \Omega ;%
%TCIMACRO{\U{211d} }%
%BeginExpansion
%\mathbb{R}
%EndExpansion
%^{d}\right) $ with $v_{h}\rightarrow 0$ in $L^{B}\left( \Omega ;%
%TCIMACRO{\U{211d} }%
%BeginExpansion
%\mathbb{R}
%EndExpansion
%^{d}\right) .$ We still represent 
let $j\in 
%TCIMACRO{\U{2124} }%
%BeginExpansion
\mathbb{Z}
%EndExpansion
^{n}$ and define $Y_{\nu ,j}:=\frac{1}{\nu }\left( j+Y\right) $ and $J_{\nu }:=\left\{
j\in 
\mathbb{Z}^{B}:\overline{Y}_{\nu ,j}\cap \overline{\Omega }\neq \emptyset \right\} $
then $\bigcup\limits_{j\in J_{\nu }} Y_{\nu ,j}\subseteq \Omega _{\frac{1%
}{\nu }}$ and $\Omega _{\frac{1}{\nu }}=\Omega \cup \Omega _{\frac{1}{\nu }%
}\backslash \Omega .$%
By \eqref{tch3}, it results
\begin{align*}
\sum\limits_{j\in J_{\nu }}\int_{Y_{\nu ,j}}f\left( \nu hx,\nabla u+\nabla
v_{h}\right) dx\leq \int_{\Omega _{\frac{1}{\nu }}}f\left( \nu hx,\nabla
u+\nabla v_{h}\right) dx \\ 
=\int_{\Omega }f\left( \nu hx,\nabla u+\nabla v_{h}\right) dx+\int_{\Omega
_{\frac{1}{\nu }}\backslash \Omega }f\left( \nu hx,\nabla u+\nabla
v_{h}\right) dx \\ 
\leq \int_{\Omega }f\left( \nu hx,\nabla u+\nabla v_{h}\right)
dx+\int_{\Omega _{\frac{1}{\nu }}\backslash \Omega } a\left( \nu hx\right) dx+M\int_{\Omega _{\frac{1}{\nu }}\backslash
\Omega }B\left( \left\vert \nabla u+\nabla v_{h}\right\vert \right) dx.
\end{align*}

But $v_{h}=0$ on $\Omega _{\frac{1}{\nu }}\backslash \Omega $ then $\nabla
v_{h}=0$ on $\Omega _{\frac{1}{\nu }}\backslash \Omega ;$ and we have 
\begin{align*}
\sum\limits_{j\in J_{\nu }}\int_{Y_{\nu ,j}}f\left( \nu hx,\nabla u+\nabla
v_{h}\right) dx\leq \int_{\Omega _{\frac{1}{\nu }}}f\left( \nu hx,\nabla
u+\nabla v_{h}\right) dx\\ 
=\int_{\Omega }f\left( \nu hx,\nabla u+\nabla v_{h}\right) dx+\int_{\Omega
_{\frac{1}{\nu }}\backslash \Omega }f\left( \nu hx,\nabla u+\nabla
v_{h}\right) dx \\ 
\leq \int_{\Omega }f\left( \nu hx,\nabla u+\nabla v_{h}\right)
dx+\int_{\Omega _{\frac{1}{\nu }}\backslash \Omega } a\left( \nu hx\right) dx+M\int_{\Omega _{\frac{1}{\nu }}\backslash
\Omega }B\left( \left\vert \nabla u\right\vert \right) dx.
\end{align*}%
Passing to liminf on $h$ we get \ 
\begin{align*}
\sum\limits_{j\in J_{\nu }}\underset{h\rightarrow +\infty }{\lim \inf }%
\int_{Y_{\nu ,j}}f\left( \nu hx,\nabla u+\nabla v_{h}\right) dx \\ 
\leq \underset{h\rightarrow +\infty }{\lim \inf }\int_{\Omega }f\left( \nu
hx,\nabla u+\nabla v_{h}\right) dx+{\mathcal L}^N( \Omega _{\frac{1}{\nu }%
}\backslash \Omega )\int_{Y} a
\left( y\right) dy+M\int_{\Omega _{\frac{1}{\nu }}\backslash \Omega }B\left(
\left\vert \nabla u\right\vert \right) dx
\end{align*}
Taking in account lemma \ref{lem1}, the properties of infima, 
and the fact that $Y_{\nu ,j}\in \mathcal{A}_{0}$, from
\begin{align*}
\inf \big\{ 
\underset{h\rightarrow +\infty }{\lim \inf }\int_{Y_{\nu ,j}}f\left( \nu
hx,\nabla u+\nabla v_{h}\right) dx:v_{h}\subseteq W^{1}L^B\left( Y_{\nu ,j};%
\mathbb{R}
^{d}\right),\\ 
v_{h}\rightarrow u\text{ in }L^{B}\left( Y_{\nu ,j};
\mathbb{R}
^{d}\right)
\big\}  \\ 
=\inf \big\{ 
\underset{h\rightarrow +\infty }{\lim \inf }\int_{Y_{\nu ,j}}f\left( \nu
hx,\nabla u+\nabla u_{h}\right) dx:u_{h}\subseteq u+W^1L^B_0\left( Y_{\nu
,j};
\mathbb{R}
^{d}\right), \\ 
u_{h}\rightarrow u\text{ in }L^{B}\left( Y_{\nu ,j};
\mathbb{R}
^{d}\right)
\big\}
\end{align*}
we get that there exists $\{w_{\nu ,j,h}\} \subset W^1L^B_0(Y_{\nu,j};\mathbb R^d)$ such that $w_{\nu,j,h}\to 0$ in $L^B(Y_{\nu,j};\mathbb R^d)$ as $h \to +\infty$, and $$\underset{h\rightarrow +\infty }{\lim \inf }\int_{Y_{\nu ,j}}f\left(
\nu hx,\nabla u+\nabla w_{\nu ,j,h}\right) dx\leq \underset{h\rightarrow
+\infty }{\lim \inf }\int_{Y_{\nu ,j}}f\left( \nu hx,\nabla u+\nabla
v_{h}\right) dx.$$ 

Fix $j\in J_{\nu },h\in 
\mathbb{N}$ consider  $\mathcal{T}_{\frac{1}{\nu }}$ on $Y_{\nu ,j},$ following
the same arguments as in \cite{ciora1} we get
\begin{align*}
\int_{Y_{\nu ,j}}f\left( \nu hx,\nabla u+\nabla w_{\nu ,j,h}\right) dx\geq
\int_{Y_{\nu ,j}\times Y}f\left( hy,\mathcal{T}_{\frac{1}{\nu }}\left(
\nabla u\right) +\mathcal{T}_{\frac{1}{\nu }}\left( \nabla w_{\nu
,j,h}\right) \left( x,y\right) \right) dxdy.
\end{align*}

\noindent Since $\mathcal{T}_{\frac{1}{\nu }}\left( \nabla w_{\nu ,j,h}\right) =\nu
\nabla _{y}\mathcal{T}_{\frac{1}{\nu }}\left( w_{\nu ,j,h}\right),$ and that, when $x$ varies almost everywhere in some $Y_{\nu,j}$, the function,$\mathcal{T%
}_{\frac{1}{\nu }}\left( \nabla w_{\nu ,j,h}\right)(x,\cdot)\in W^1L^B_0(Y;\mathbb R^d)$, it results that $\mathcal{T%
}_{\frac{1}{\nu }}\left( \nabla w_{\nu ,j,h}\right) \in L_{D_{0y}}^{B}\left(
Y_{\nu ,j}\times Y;\mathbb{R}
^{d}\right) $. 
Thus
\begin{align*}
\int_{Y_{\nu ,j}}f\left( \nu hx,\nabla u+\nabla w_{\nu ,j,h}\right) dx\geq 
\\ 
\inf \left\{ \int_{Y_{\nu ,j}\times Y}f\left( hy,\mathcal{T}_{\frac{1}{\nu }
}\left( \nabla u\left( x,y\right) \right) +\nabla _{y}V\left( x,y\right)
\right) dxdy:V\in L_{D_{0y}}^{B}\left( Y_{\nu ,j}\times Y;
\mathbb R^{d}\right) \right\}.
\end{align*}
Note that $\bigcup\limits_{j\in J_{\nu }} Y_{\nu ,j}\cup \mathcal{N}%
=\Omega _{\frac{1}{\nu }}$ with $\mathcal L^N(\mathcal{N})=0$, with the $Y_{\nu
,j}\subseteq \Omega _{\frac{1}{\nu }}$ pairwise disjoints, we have
\begin{align*}
\sum\limits_{j\in J_{\nu }}\inf \left\{ \int_{Y_{\nu ,j}\times Y}f\left(
hy,\mathcal{T}_{\frac{1}{\nu }}\left( \nabla u\left( x,y\right) \right)
+\nabla _{y}V\left( x,y\right) \right) dxdy:V\in L_{D_{0y}}^{B}\left( Y_{\nu
,j}\times Y;
\mathbb{R}
%EndExpansion
^{d}\right) \right\} \\ 
\geq \inf \left\{ \int_{\Omega _{\frac{1}{\nu }}\times Y}f\left( hy,%
\mathcal{T}_{\frac{1}{\nu }}\left( \nabla u\left( x,y\right) \right) +\nabla
_{y}V\left( x,y\right) \right) dxdy:V\in L_{D_{0y}}^{B}\left( \Omega \times
Y;\mathbb{R}
^{d}\right) \right\}  \\ 
\geq \inf \left\{ \int_{\Omega \times Y}f\left( hy,\mathcal{T}_{\frac{1}{\nu }%
}\left( \nabla u\left( x,y\right) \right) +\nabla _{y}V\left( x,y\right)
\right) dxdy:V\in L_{D_{0y}}^{B}\left( \Omega \times Y;
\mathbb{R}
^{d}\right) \right\} .
\end{align*}
It follows that 
\begin{align*}
\underset{h\rightarrow +\infty }{\lim \inf }\int_{\Omega }f\left( \nu
hx,\nabla u+\nabla v_{h}\right) dx+\mathcal L^N(\Omega _{\frac{1}{\nu }%
}\backslash \Omega) \int_{Y}a
\left( y\right) dy+M\int_{\Omega _{\frac{1}{\nu }}\backslash \Omega }B\left(
\left\vert \nabla u\right\vert \right) dx \\ 
\geq \sum\limits_{j\in J_{\nu }}\underset{h\rightarrow +\infty }{\lim \inf 
}\inf \left\{ \int_{Y_{\nu ,j}\times Y}f\left( hy,\mathcal{T}_{\frac{1}{\nu }%
}\left( \nabla u\left( x,y\right) \right) +\nabla _{y}V\left( x,y\right)
\right) dxdy:V\in L_{D_{0y}}^{B}\left( Y_{\nu ,j}\times Y;
\mathbb{R}
^{d}\right) \right\}  \\ 
=\underset{h\rightarrow +\infty }{\lim }\sum\limits_{j\in J_{\nu }}\inf
\left\{ \int_{Y_{\nu ,j}\times Y}f\left( hy,\mathcal{T}_{\frac{1}{\nu }%
}\left( \nabla u\left( x,y\right) \right) +\nabla _{y}V\left( x,y\right)
\right) dxdy:V\in L_{D_{0y}}^{B}\left( Y_{\nu ,j}\times Y;
\mathbb{R}
^{d}\right) \right\} \\ 
\geq \underset{h\rightarrow +\infty }{\lim }\inf \left\{ \int_{\Omega
\times Y}f\left( hy,\mathcal{T}_{\frac{1}{\nu }}\left( \nabla u\left(
x,y\right) \right) +\nabla _{y}V\left( x,y\right) \right) dxdy:V\in
L_{D_{0y}}^{B}\left( \Omega \times Y;
\mathbb{R}
^{d}\right) \right\}.
\end{align*}

Since $%
\mathcal L^N( \Omega _{\frac{1}{\nu }}\backslash \Omega) \rightarrow
0$ as $\nu \rightarrow +\infty ,$ by the absolute continuity of the Lebesgue integral we have
\begin{align*}
\liminf_{\nu \to +\infty}\underset{h\rightarrow +\infty }{\lim \inf }\int_{\Omega }f\left( \nu
	hx,\nabla u+\nabla v_{h}\right) dx \\
	\geq
 \liminf_{\nu \to +\infty}\underset{h\rightarrow +\infty }{\lim }\inf \left\{ \int_{\Omega
	\times Y}f\left( hy,\mathcal{T}_{\frac{1}{\nu }}\left( \nabla u\left(
x,y\right) \right) +\nabla _{y}V\left( x,y\right) \right) dxdy:V\in
L_{D_{0y}}^{B}\left( \Omega \times Y;
\mathbb{R}
^{d}\right) \right\}.
\end{align*}
Consequently the proof will be concluded if we show the following: 
\begin{align*}
\underset{\nu \rightarrow +\infty }{\lim \inf }\underset{h\rightarrow
+\infty }{\lim }\inf \left\{ \int_{\Omega \times Y}f\left( hy,\mathcal{T}_{%
\frac{1}{\nu }}\left( \nabla u\left( x\right) \right) +\nabla _{y}V\left(
x,y\right) \right) dxdy:V\in L_{D_{0y}}^{B}\left( \Omega \times Y;%
\mathbb{R}
^{d}\right) \right\} \\ 
\underset{h\rightarrow +\infty }{\geq \lim \inf }\inf \left\{ \int_{\Omega
\times Y}f\left( hy,\left( \nabla u\left( x\right) \right) +\nabla
_{y}V\left( x,y\right) \right) dxdy:V\in L_{D_{0y}}^{B}\left( \Omega \times
Y;
\mathbb R^{d}\right) \right\}. 
\end{align*}

Taking in account definition of infimum, and properties of
unfolding we get the existence of a constat $k\left( u\right) $ (depending only on 
$u$) such that:
\begin{align*}
\inf \left\{ \int_{\Omega \times Y}f\left( hy,\mathcal{T}_{\frac{1}{\nu }%
}\left( \nabla u\left( x\right) \right) +\nabla _{y}V\left( x,y\right)
\right) dxdy:V\in L_{D_{0y}}^{B}\left( \Omega \times Y;
\mathbb{R}
^{d}\right) \right\}  \\ 
= \inf \big\{ 
\int_{\Omega \times Y}f\left( hy,\mathcal{T}_{\frac{1}{\nu }}\left( \nabla
u\left( x\right) \right) +\nabla _{y}V\left( x,y\right) \right) dxdy: \\ 
V\in L_{D_{0y}}^{B}\left( \Omega \times Y;
\mathbb{R}
^{d}\right) ,\int_{\Omega \times Y}B\left( \left\vert \mathcal{T}_{\frac{1}{
\nu }}\left( \nabla u\right) +\nabla _{y}V\right\vert \right) dxdy\leq
k\left( u\right)
\big\}.
\end{align*}
Note that $\left\Vert \mathcal{T}_{\frac{1}{\nu }}\left( \nabla u\left(
x\right) \right) +\nabla _{y}V\left( x,y\right) \right\Vert _{L^{B}\left(
\Omega \times Y;%
\mathbb{R}
^{d}\right) }\leq \int_{\Omega \times Y}B\left( \left\vert \mathcal{T}_{%
\frac{1}{\nu }}\left( \nabla u\right) +\nabla _{y}V\right\vert \right) dxdy$%
\ whenever $\left\Vert \mathcal{T}_{\frac{1}{\nu }}\left( \nabla u\left(
x\right) \right) +\nabla _{y}V\left( x,y\right) \right\Vert _{L^{B}\left(
\Omega \times Y;
\mathbb{R}
^{d}\right) }\geq 1.$
For $\ V\in L_{D_{0y}}^{B}\left( \Omega \times Y;
\mathbb{R}
^{d}\right) ,$ with $\int_{\Omega \times Y}B\left( \left\vert \mathcal{T}_{%
\frac{1}{\nu }}\left( \nabla u\right) +\nabla _{y}V\right\vert \right) dxdy$ 
$\leq k\left( u\right) $ we have, making the same type of estimates as in \eqref{eq4},: 
\begin{align*}
\int_{\Omega \times Y}f\left( hy,\mathcal{T}_{\frac{1}{\nu }}\left( \nabla
u\left( x\right) \right) +\nabla _{y}V\left( x,y\right) \right) dxdy\\ 
\geq \int_{\Omega \times Y}f\left( hy,\nabla u\left( x\right) +\nabla
_{y}V\left( x,y\right) \right) dxdy-C_{1\varepsilon }\int_{\Omega }B\left(
\left\vert \mathcal{T}_{\frac{1}{\nu }}\left( \nabla u\right) -\nabla
u\right\vert \right) dx \\ 
-\varepsilon \int_{\Omega }C_{2}\widetilde{B}\left( 1\right) dx-\varepsilon
C_{3}C_{4}\int_{\Omega }\left( B\left( 1\right) dx\right)  \\ 
-\varepsilon C_{3}C_{4}\int_{\Omega }B\left( \left\vert \mathcal{T}_{\frac{1%
}{\nu }}\left( \nabla u\left( x\right) \right) +\nabla _{y}V\left(
x,y\right) \right\vert \right) dx \\ 
-\varepsilon C_{3}C_{4}\int_{\Omega }B\left( \left\vert \nabla u\left(
x\right) +\nabla _{y}V\left( x,y\right) \right\vert \right) dx\\ 
\geq \int_{\Omega \times Y}f\left( hy,\nabla u\left( x\right) +\nabla
_{y}V\left( x,y\right) \right) dxdy-C_{1\varepsilon }\int_{\Omega }B\left(
\left\vert \mathcal{T}_{\frac{1}{\nu }}\left( \nabla u\right) -\nabla
u\right\vert \right) dx \\ 
-\varepsilon C_{3}C_{4}k\left( u\right) -\varepsilon C_{3}C_{4}\int_{\Omega
}B\left( \left\vert \nabla u\left( x\right) +\nabla _{y}V\left( x,y\right)
\right\vert \right) dx\end{align*}%
Therefore for $\nu ,h\in 
\mathbb{N}$ we get 
\begin{align*}
\inf \big\{ 
\int_{\Omega \times Y}f\left( hy,\mathcal{T}_{\frac{1}{\nu }}\left( \nabla
u\left( x\right) \right) +\nabla _{y}V\left( x,y\right) \right) dxdy: 
V\in L_{D_{0y}}^{B}\left( \Omega \times Y;\mathbb{R}
^{d}\right)\big\}  \\ 
\geq \inf \left\{\int_{\Omega \times Y}f\left(hy,\nabla u\left( x\right)
+\nabla_{y}V\left( x,y\right) \right) dxdy:V\in L_{D_{0y}}^{B}\left( \Omega
\times Y;
\mathbb{R}
^{d}\right) \right\}  \\ 
-C_{1\varepsilon }\int_{\Omega}B\left(\left\vert \mathcal{T}_{\frac{1}{
\nu }}\left( \nabla u\right) -\nabla u\right\vert \right) dx 
-\varepsilon \int_{\Omega }C_{2}\widetilde{B}\left( 1\right) dx-\varepsilon
C_{3}C_{4}\int_{\Omega }B\left( 1\right) dx \\ 
-\varepsilon C_{3}C_{4}\int_{\Omega }B\left( \left\vert \mathcal{T}_{\frac{1%
}{\nu }}\left( \nabla u\left( x\right) \right) +\nabla _{y}V\left(
x,y\right) \right\vert \right) dx \\ 
-\varepsilon C_{3}C_{4}\int_{\Omega }B\left( \left\vert \nabla u\left(
x\right) +\nabla _{y}V\left( x,y\right) \right\vert \right) dx \\ 
\geq \inf \left\{ \int_{\Omega \times Y}f\left( hy,\nabla u\left( x\right)
+\nabla _{y}V\left( x,y\right) \right) dxdy:V\in L_{D_{0y}}^{B}\left( \Omega
\times Y;
\mathbb{R}^{d}\right) \right\}  \\
\geq -C_{1\varepsilon }\int_{\Omega }B\left( \left\vert \mathcal{T}_{\frac{1
}{\nu}}\left(\nabla u\right) -\nabla u\right\vert \right) dx-\varepsilon
C_{3}C_{4}k\left( u\right) -\varepsilon C_{3}C_{4}C_{5}k\left( u\right) \\ 
-\varepsilon C_{3}C_{4}\int_{\Omega }B\left( \left\vert \nabla u\left(
x\right) -\mathcal{T}_{\frac{1}{\nu }}\left( \nabla u\right) \right\vert
\right) dx.  
\end{align*}
Taking the limit first on $\nu \to +\infty$ and as $\varepsilon \to 0$, and recalling \eqref{eqtoest}, the proof is concluded.
\end{proof}
\begin{proof}[Proof of Theorem \ref{main1}]
It is a consequence of lemmas \ref{lem1}, \ref{lem3.2}, \ref{lem3.6}, and propositions \ref{prop2.8} and \ref{ar3}.
\end{proof}

\textbf{Acknowledgements.}  Fotso Tachago is grateful to  Department of Basic and Applied Science for Engineering of Sapienza - University of Rome for its kind hospitality, during the preparation of this work. He also acknowledges the support received by International Mathematical Union, through IMU grant 2024.  
E.~Zappale acknowledges the support of the project
``Mathematical Modelling of Heterogeneous Systems (MMHS)",
financed by the European Union - Next Generation EU,
CUP B53D23009360006, Project Code 2022MKB7MM, PNRR M4.C2.1.1.
 She is a member of the Gruppo Nazionale per l'Analisi Matematica, la Probabilit\`a e le loro Applicazioni (GNAMPA) of the Istituto Nazionale di Alta Matematica ``F.~Severi'' (INdAM), and  also acknowledges partial funding from the GNAMPA Project 2023 \emph{Prospettive nelle scienze dei materiali: modelli variazionali, analisi asintotica e omogeneizzazione}.

\smallskip

\end{document}